\documentclass[12pt]{amsart}

\usepackage{amsfonts}
\usepackage{amssymb}
\usepackage{amsthm}
\usepackage{amsmath}
\usepackage{enumitem}
\usepackage{graphicx}
\newtheorem{theorem}{Theorem}[section]
\newtheorem{lemma}[theorem]{Lemma}
\newtheorem{remark}[theorem]{Remark}
\newtheorem{proposition}[theorem]{Proposition}

\newtheorem{definition}[theorem]{Definition}
\numberwithin{equation}{section}

\newcommand{\bS}{{\mathbb{S}}}

\newcommand{\cT}{{\mathcal{T}}}

\begin{document}
\title{Equivariant min-max theory}
\author{Daniel Ketover}\address{Department of Mathematics\\Princeton
University\\Princeton, NJ 08544}
\thanks{The author was partially supported by NSF Grant DMS-11040934 and an NSF Postdoctoral Research fellowship as well as ERC-2011-StG-278940.}
 \email{dketover@math.princeton.edu}
\maketitle
\begin{abstract}
We develop an equivariant min-max theory as proposed by Pitts-Rubinstein in 1988 and then show that it can produce many of the known minimal surfaces in $\mathbb{S}^3$ up to genus and symmetry group.  We also produce several new infinite families of minimal surfaces in $\mathbb{S}^3$ proposed by Pitts-Rubinstein. These examples are doublings and desingularizations of stationary integral varifolds in $\bS^3$. 
\end{abstract}

\section{Introduction}
Constructing embedded minimal surfaces in a given $3$-manifold is a difficult problem.  If there are incompressible  surfaces in the manifold one can minimize area in the isotopy class of such a surface to produce a minimal surface by a result of Meeks-Simon-Yau \cite{msy}. The only technique that works in full generality with no assumptions on the metric or manifold is the min-max technique of Almgren-Pitts and later refined by Simon-Smith \cite{smith}.  In this approach one considers sweep-outs of the manifold and the smallest slice needed to ``pull over" the manifold is a smooth embedded minimal surface.   Recently, Marques-Neves have used this technique and higher parameter families to construct infinitely many minimal surfaces in manifolds with positive Ricci curvature \cite{mn}.  In full generality however it is very hard to control the genus of the limiting minimal surfaces beyond the genus bounds proved in \cite{ketover} (which built on earlier work of Simon-Smith \cite{smith} and De Lellis-Pellandini \cite{dep}). See Colding-De Lellis \cite{cd} for an exposition of the min-max theory.   

If the manifold has some symmetries, one can hope to control the genus of the minimal surfaces produced from a min-max procedure.  Toward that end, in the late 80s Pitts-Rubinstein \cite{pr} proposed considering the situation when one has a finite group of isometries $G$ acting on a $3$-manifold.  Then one can consider sweepouts where each sweepout surface is preserved by the group and run a min-max procedure that only includes such equivariant surfaces.  They claimed one should be able to produce $G$-equivariant minimal surfaces this way.  The main result of this paper (Theorem \ref{main}) is that this procedure conjectured by Pitts-Rubinstein in fact works.  

It was further proved in \cite{ketover} that after performing finitely many neck-pinch surgeries on the the min-max sequence, the remaining components align themselves as covering about the limiting minimal surfaces with the expected multiplicities.  Since we are restricted here to equivariant sweepouts, any neckpinch must {\it also} be equivariant.   This severely limits the type of degeneration that occurs and will allow us to control the genus of the limiting minimal surfaces in most cases.  In fact, we will show that in round $\mathbb{S}^3$, many (if not all) of the known embedded minimal surfaces can be generated from our equivariant min-max process.  

One might object that only considering equivariant deformations one should not be able to produce a surface critical with respect to \emph{all} variations.  But there is a very old principle in mathematics and mathematical physics formulated by Palais \cite{palais} called the ``Principle of Symmetric Criticality." It says roughly speaking that if one has a functional on a space with a symmetry, and there is a symmetric point in the space where the functional has zero derivative in directions that have the symmetry, then the point is actually critical with respect to all deformations.   This is why one can expect to produce globally minimal surfaces when one is only looking at sweepouts and comparison surfaces that are also equivariant.\\
\indent  As a simple illustration of this principle, if one has a smooth $G$-equivariant surface $\Sigma$ and the area of $\Sigma$ has zero variation among equivariant deformations, then we can easily show it must be a smooth minimal surface (see Theorem 1 in \cite{hsianglawson}).
Indeed by the first variation formula, for any vector field $V$ defined on a neighborhood of $\Sigma$ we have
\begin{equation}\label{f}
\delta\Sigma(V)=-\int_\Sigma \langle X,H\rangle.
\end{equation}
Since $\Sigma$ is equivariant, so is the vector $H$ and we can thus plug $H$ into \eqref{f} and obtain by stationarity among equivariant deformations: $$0=\delta\Sigma(H)=-\int_\Sigma |H|^2,$$ whence $H=0$ identically.  

The idea of finding minimal surfaces in a singular quotient manifold $M/G$ and then lifting to produce a minimal surface in $M$ goes at least as far back as Hsiang-Lawson \cite{hsianglawson}.  This is how W.Y. Hsiang constructed examples of non-equatorial minimal spheres embedded in high dimensional spheres \cite{hsiang}.  In previous work the analysis amounts to finding geodesics on the quotient space.  To our knowledge our results are the first where one allows $M/G$ to have dimension greater than $2$.

We first give a few preliminary facts about group actions in order to state our results.  Throughout this paper $G$ will denote a finite group of orientation-preserving isometries acting on a $3$-manifold $M$.
For any $x\in M$ we first define the \emph{isotropy subgroup} $G_x$ at $x$ as:
$$G_x=\{g\in G\;\;|\;\;gx = x\}$$
We then define the \emph{singular locus} of the group action as points with nontrivial isotropy subgroup:
$$\tilde{\mathcal{S}}=\{ x\in M\;\;|\;\;G_x\neq\{e\}\}$$
Henceforth we will restrict our attention to groups $G$ such that $M/G$ is an orientable orbifold without boundary. Let $\pi:M\rightarrow M/G$ be the projection map and set $\mathcal{S}=\pi(\tilde{\mathcal{S}})$.  In this case, $\mathcal{S}$ has the structure of a trivalent graph (see for instance Section 2.3.4 in \cite{KL}).  That is, we can stratify the set $\mathcal{S}$: $$\mathcal{S}=\mathcal{S}_0\cup\mathcal{S}_1,$$ where $\mathcal{S}_1$ consists of geodesic segments that connect (potentially) in the finitely many vertices comprising the set $\mathcal{S}_0$.   Three segments meet at each point in $\mathcal{S}_0$.  Along each segment in $\mathcal{S}_1$ the isotropy subgroup is $\mathbb{Z}_n$ for some $n\in\mathbb{Z}^+$.  Points in the singular set $\mathcal{S}_0$ look locally like a cone over a two-dimensional spherical orbifold.  At such a vertex in $\mathcal{S}_0$ where three edges are meeting (indexed by their isotropy groups of orders $n_1$, $n_2$, and $n_3$), we have as possibilities: $$\mathbb{D}_{2n} =(2,2,n),\;\; T_{12}=(2,3,3),\;\;O_{24}=(2,3,4),\;\; I_{60}=(2,3,5),$$ which correspond to the classification of $2$-d spherical orbifolds. It will be convenient to set $\tilde{\mathcal{S}}_0=\pi^{-1}(\mathcal{S}_0)$ and $\tilde{\mathcal{S}}_1=\pi^{-1}  (\mathcal{S}_1)$.  It is important to remember that $\tilde{\mathcal{S}}$ will not have the same trivalent structure as its projection $\mathcal{S}$.  Let us call a smooth connected segment in $\tilde{\mathcal{S}}_1$ an \emph{arc of constant isotropy} if the isotropy group at all points is the same and the segment is maximal with respect to this property.  If not a closed curve, such an arc of constant isotropy has points in $\tilde{\mathcal{S}}_0$ as its endpoints.\\ \indent
Aside from a group action, the other notion we need to state our main result is that of a $G$-sweepout:
\begin{definition}\label{gsweepout}
If $M$ is a closed $3$-manifold and $G$ a finite group of orientation-preserving isometries.  A {\it $G$-sweepout} of $M$ is a family of closed sets $\{\Sigma_t\}_{t=0}^1$, continuously varying in the Hausdorff topology such that:
\begin{enumerate}[label=\roman*.,topsep=2pt,itemsep=-1ex,partopsep=1ex,parsep=1ex]
\item $\Sigma_t$ is a smooth embedded surface for $0<t<1$ varying smoothly
\item $\Sigma_0$ and $\Sigma_1$ are $1$-d graphs in $M$
\item Each $\Sigma_t$ is $G$-equivariant, i.e. $g(\Sigma_t)=\Sigma_t$ for $0\leq t\leq 1$ and all $g\in G$
\item Only the slices $\Sigma_0$ and $\Sigma_1$ intersect $\tilde{\mathcal{S}}_0$.
\item The surfaces $\Sigma_t$ for $0<t<1$ intersect $\tilde{\mathcal{S}}$ transversally
\end{enumerate}
\end{definition}
\begin{remark}\normalfont
In fact, v. forces the surfaces to intersect $\tilde{\mathcal{S}}_1$ orthogonally.  It follows that for $0<t<1$, each surface $\Sigma_t$ intersects each arc of constant isotropy a fixed number of times (see Lemma \ref{intersect3}).  Note that iv. is already implied by Lemma \ref{actions} and the fact that the surfaces $\Sigma_t$ are smooth.
\end{remark}
\noindent

Given such a $G$-sweepout $\{\Sigma_t\}_{t=0}^1$ we may define the $G$-equivariant saturation $\Pi=\Pi_{\{\Sigma_t \}}$ identically as in \cite{cd} except where ``isotopy" is replaced by ``equivariant isotopy" (see Definition \ref{eqisotopy}).  We can then define the min-max width (where throughout this paper $\mathcal{H}^2$ denotes $2$-dimensional Hausdorff measure):
\begin{equation}\label{inf}
W^G_\Pi = \inf_{\{\Lambda_t\}\in\Pi}\sup_{t\in [0,1]}\mathcal{H}^2(\Lambda_t).
\end{equation}

It follows easily as in Proposition 1.4 in \cite{cd} that $W^G_\Pi >0$ (one is restricting to equivariant isotopies so  $W^G_\Pi$ is at least as large as the non-equivariant width considered in \cite{cd} using all isotopies).

We can then consider a sequence of sweepouts $\{\Sigma_t\}^i$ the area of whose maximal slice converges to $W^G_\Pi$.  From $\{\Sigma_t\}^i$ we may then choose a sequence of slices $\Sigma_i:=\Sigma_{t_i}^i$ with area converging to $W^G_\Pi$.  Such a sequence of surfaces we will call a \emph{min-max sequence}. 

We can now state our main results. Our following theorem on equivariant min-max was announced in some form in 1988 by Pitts-Rubinstein \cite{pr} but the author is not aware of a published proof.
\begin{theorem}\label{main} Let $M$ be a closed orientable Riemannian $3$-manifold and let $G$ be a finite group of orientation preserving isometries acting on $M$ such that $M/G$ is an orientable orbifold without boundary. 
Let $\{\Sigma_t\}_{t=0}^1$ be a $G$-sweepout of $M$ by surfaces of genus $g$.  Then we have the following:
\begin{enumerate}[label=\alph*.] \item There exists a min-max sequence $\Sigma_j$ converging as varifolds to $\Gamma=\sum_1^k n_i\Gamma_i$, where $\Gamma_i$ are smooth embedded pairwise disjoint minimal surfaces and $n_i$ are positive integers.  In particular, $M$ contains an embedded $G$-equivariant minimal surface.  
\item $\sum_{i=1}^k n_i\mathcal{H}^2(\Gamma_i)= W^G_\Pi$
\item For $j$ large enough, after performing finitely many $G$-equivariant neck-pinch surgeries on $\Sigma_j$ and discarding some components, each remaining component of $\Sigma_j$ is isotopic to one of the $\Gamma_i$ or to a double cover.  After this surgery process, for each $i$, there are $n_i$ components of the min-max sequence isotopic to $\Gamma_i$ if $\Gamma_i$ is orientable, and $n_i/2$ components isotopic to a double cover if $\Gamma_i$ is non-orientable.

\item Item c) implies the genus bound with multiplicity:
\begin{equation}\label{multi}
\sum_{i\in\mathcal{O}} n_ig(\Sigma_i)+\sum_{i\in\mathcal{N}}\frac{n_i}{2}(g(\Sigma_i)-1)\leq g,
\end{equation}
where $\mathcal{O}$ denotes the subcollection of $\Gamma_i$ that are orientable, and $\mathcal{N}$ denotes the subcollection that are non-orientable.
\item If $x\in\tilde{\mathcal{S}}_1\cap\Gamma_i$ and $G_x\neq\mathbb{Z}_2$ then $\tilde{\mathcal{S}}$ intersects $\Gamma_i$ orthogonally at $x$.
\item If $x\in\tilde{\mathcal{S}}_1\cap\Gamma_i$ and $G_x=\mathbb{Z}_2$ and either $\Gamma_i$ is orthogonal to $\tilde{\mathcal{S}}$ or else $\tilde{\mathcal{S}}$ is tangent to $\Gamma_i$ at $x$, in which case $\Gamma_i$ has even multiplicity and $\Gamma_i$ contains the arc of constant isotropy containing $x$.
\item $\Gamma$  can only intersect $\tilde{\mathcal{S}}_0$ at a point $x$ with isotropy group $\mathbb{D}_n$.   A component of $\Gamma$ containing such a point has even multiplicity.  In this case, such a component contains $2n$ of the arcs of isotropy $\mathbb{Z}_2$ intersecting at $x$ (if $n=2$, there are three possible such pairings, otherwise a unique set of such arcs).
\end{enumerate}
\end{theorem}
\noindent
A few remarks about Theorem \ref{main} are in order:
\begin{remark}\normalfont
Theorem \ref{main} is trivial if $G$ acts freely on $M$, for in this case $M/G$ is a smooth manifold and one can run the min-max theory relative to a Heegaard splitting of $M/G$ and lift the resulting minimal surface up to $M$ to produce a $G$-equivariant minimal surface.  Thus the content of the theorem is in the situation when $G$ acts non-freely and where $M/G$ is an orbifold.  \end{remark}
\begin{remark}\normalfont
Even though only one parameter is used, because of the symmetry imposed on the sweepouts, the surfaces we produce will have high Morse index in general.  It would be interesting to determine this index.  The surfaces we produce should have equivariant index $1$.
\end{remark}
\begin{remark}\label{typeofneck}\normalfont
Let us explain the meaning of \emph{$G$-equivariant neck-pinch surgeries} in the statement of Theorem \ref{main}c.  Such neck-pinches consist of two varieties:  The first type is an \emph{ordinary neck-pinch} which is the removal of a cylinder from a surface and attachment of two disks so that the union of disk and cylinder bound a ball, all disjoint from the singular set $\tilde{\mathcal{S}}$.  If such a neck-pinch is performed, there are $|G|$ isometric copies of the neck-pinch which must also be performed concurrently to preserve equivariance.  The second type of neck-pinch we call a \emph{$\mathbb{Z}_n$-neckpinch}.  Here the cylinder removed is centered around an arc of isotropy $\mathbb{Z}_n$, and the two disks one adds in each intersect the singular arc once orthogonally.  The ball bounded by the two disks and cylinder contain only a piece of the arc of $\mathbb{Z}_n$ isotropy. In this case, there $|G|/n$ isometric copies of the neck-pinch which must be performed concurrently.
\end{remark}
\begin{remark} \normalfont Regarding items e,f, when no isotropy subgroups are $\mathbb{Z}_2$ we can prevent the min-max sequence from becoming tangent (``creasing" into) to the singular set.  However, if the isotropy subgroup is $\mathbb{Z}_2$ at a point, the min-max sequence {\it can} potentially press into the singular set (though we know of no instance where this actually happens). 
To see the relevance of the order of the cyclic group, consider $G=\mathbb{Z}_n$ acting on $\mathbb{R}^3$ by rotations of angle $2\pi/n$ about the z-axis.  When $n\neq 2$, the only plane passing through the origin that is $\mathbb{Z}_n$-equivariant is the horizontal plane: $\{(x,y,z)\;|\; z=0\}$ which is perpendicular to the singular set.  The problem is that for $n=2$, any planes of the form $\{(x,y,z) \;|\; ax+by=c\}$ is also $\mathbb{Z}_2$-equivariant. Thus a stable equivariant minimal surface can {\it contain} the singular set.  Worse yet, for each $\epsilon>0$ one has a $\mathbb{Z}_2$-equivariant stable surface: $$\Sigma_\epsilon = \{y=\epsilon\}\cup\{y=-\epsilon\}$$ that is disjoint from the singular set and yet as $\epsilon\rightarrow 0$ these surfaces converge with multiplicity two to the plane $\{y=0\}$ containing $\mathcal{S}$.  Thus in principle a min-max sequence could push with even multiplicity into the singular set.  

This behavior along curves of $\mathbb{Z}_2$ isotropy is not a problem for the regularity theory (since the failure of compactness of stable $\mathbb{Z}_2$-equivariant surfaces amounts to the formation of multiplicity, which anyway is allowed in the theory) but it is a problem when we want to control the genus of the min-max surface.   For instance in order to double the Clifford torus, the relevant group contains an involution that sends $(z,w)$ to $(w,z)$ in $\mathbb{C}^2$ and creates a curve contained in a Clifford torus with isotropy subgroup $\mathbb{Z}_2$. The min-max sequence may therefore converge to the Clifford torus with multiplicity two.  In joint work with F. C. Marques and A. Neves \cite{KMN} we introduce the ``catenoid estimate" to deal with this issue that arises in the min-max approach to doubling constructions.
\end{remark}

\begin{figure}
\includegraphics[scale=.3] {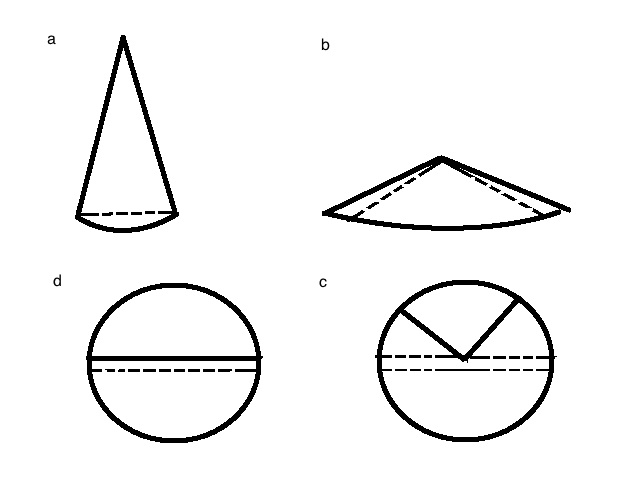}
\caption{In a) geodesics avoid the singular point when the cone angle is small.  In b) and c) we see that in wide-brimmed cones, length-minimizing geodesics can pass through the singular point.  In d) we see how when the cone angle is $\pi$, stable geodesics can converge into the cone point to give a degenerate geodesic with multiplicity 2.}
\end{figure}

Since we will see that min-max sequences are approximated by stable minimal surfaces, another way to see why the sequences cannot crease into the singular set when the isotropy group is not $\mathbb{Z}_2$ is to consider geodesics on the two-dimensional cones with cone angle $\theta$.  If the cone angle is very close to $2\pi$, stable geodesics pass through the singular point.  But we are only interested in the orbifold regime where $\theta=2\pi/n$ for some $n\in\mathbb{Z}$.  If $\theta<\pi$, stable geodesics always avoid the singular point.  The case $\theta=\pi$ is the borderline case where there exists a degenerate geodesic going from the base to the tip of the cone and back (see Figure 1).  The key point to take away is that $\mathbb{Z}_2$ area minimizing surfaces do satisfy a compactness theorem provided one allows the limit to have multiplicity. 

After proving existence and regularity of equivariant minimal surfaces, we then apply our construction to the study of minimal surfaces in round $\mathbb{S}^3$.   The classical minimal surfaces in $\mathbb{S}^3$ are the equatorial two-spheres and the Clifford torus and for a long time these were the only known surfaces.   Then in 1970 Lawson produced a minimal surfaces of every genus \cite{lawson}.  His technique was to study a symmetry group acting on the sphere, solve the Plateau problem for a polygon inscribed in the fundamental domain of the group action, and reflect about the edges to produce a closed embedded surface.   Karcher-Pinkall-Sterling \cite{kps} used the same technique to produce nine minimal surfaces associated with the Platonic solids (their method was somewhat different in that in each fundamental domain they solve a free-boundary problem rather than Plateau problem for a fixed quadrilateral).   More recently Choe-Soret \cite{choesoret} used Lawson's technique to produce new minimal surfaces.  The only other technique that has been successful is that of Kapouleas-Yang \cite{kapouleas} who used gluing techniques to desingularize two nearby Clifford tori by connecting them via many catenoidal necks. The gluing method has been applied also by Wiygul \cite{wigul1} who constructed minimal embeddings resembling ``stacks" of several nearby Clifford tori.  Finally Kapouleas \cite{kapouleas2} recently succeeded in doubling the equatorial $2$-sphere.  See Brendle's survey \cite{brendle} for more discussion of these results.

We will give a min-max construction of many of these surfaces and then construct eight new infinite families that resemble doublings and desingularizations of stationary varifolds in the round $3$-sphere.  These examples all appear in the original table in the announcement \cite{pr}.

Our methods are entirely different from Kapouleas' and Lawson's in that they are variational and global in nature -- we are simply doing Morse theory on the space of equivariant surfaces in a given manifold. The Morse-theoretic approach is  geometrically very natural, though the global nature of the sweepouts required could be a disadvantage in some cases as one cannot restrict to a tubular neighborhood of the surfaces one wants to double or desingularize. 

One definite advantage of our methods is that while in gluing constructions one always has to take the genus inserted along desingularizing curves to be large, in the min-max setting there is no such restriction. We expect that most if not all embedded desingularizations and doublings can be given a variational interpretation. \\

The organization of this paper is as follows.  In Section 2 we outline the main issues involved in the construction.  In Section  3 we will prove a weak version in the setting of Geometric Measure Theory of the symmetric critical points principle.  This will allow us to use the ``pull-tight" argument of Colding-De Lellis to produce an equivariant stationary varifold. In Section 4 we will prove the regularity of the stationary varifold produced by the min-max procedure.  In Section 5 we address the topology of the limiting minimal surfaces.  Finally in Section 6 we turn to examples and give new constructions of old minimal surfaces in $\mathbb{S}^3$ and produce many new examples.
\\
\\
\noindent
{\it Acknowledgements:  I am grateful to Renato Bettiol for pointing out the work \cite{palais}, Baris Coskunuzer for some conversations and Brian White for some discussion of Proposition \ref{minimizing}.  I also thank Fernando Marques, Andr\'e Neves and Toby Colding for their interest in this work.} 
\section{Main issues}
We explain the main points and difficulties of the construction in a more technical way.  We will be only working with isotopies that are $G$-equivariant.  The formal ``pull-tight" procedure of Colding-De Lellis allows one to work in this restricted family and produce a $G$-stationary varifold.   We will prove that $G$-stationary varifolds are in fact stationary which is a weak version of the Principle of Symmetric Criticality explained above.  Thus we can find min-max sequences converging to a stationary varifold $\Gamma_\infty$.

We then prove that our min-max sequence is almost minimizing (in the sense introduced by Almgren \cite{A} and later by Pitts) in annuli.  This is after all an abstract purely combinatorial argument using only the metric space property of the ambient space.  Thus by the regularity theory in Colding-De Lellis \cite{cd} we immediately obtain regularity away from $\tilde{\mathcal{S}}$.  But a stationary varifold that is smooth away from a one-dimensional set can be quite far from being smooth -- for instance two planes intersecting orthogonally in $\mathbb{R}^3$ is $\mathbb{Z}_4$-equivariant and smooth away from a line.  

The point is that we {\it do} have the almost minimizing property in annuli centered around points in the singular locus of the group action $\tilde{\mathcal{S}}$. The only difference is that in such annuli ``almost minimizing" means restricted to isotopies that are equivariant.   We can minimize among equivariant isotopies in such annuli to produce a $G$-equivariant minimal surface $V_j$ that is stable among equivariant deformations ($G$-stable).  But we will prove that $G$-stability implies stability among all variations as long as the surfaces in the sweepout intersect $\tilde{\mathcal{S}}$ transversally, which we are assuming (see items iv) and v) in Definition \ref{gsweepout}).  This implication uses standard facts about the first eigenfunction of Schr\"odinger operators.  Thus locally we obtain stable replacements $V_j$ and one can use the classical estimates to Schoen \cite{schoen} to produce a convergent subsequence, a ``replacement" $V_\infty$ for $\Sigma_\infty$ in an annulus around the singular set.  This is morally what is preventing the min-max sequence from ``creasing" to form an ``X" with multiplicity 2.  It then follows as in \cite{cd} that a stationary varifold that has smooth replacements in the above sense is itself a smooth minimal surface.

We also need that the $G$-stable replacements $V_j$ are in fact smooth minimal surfaces.  The way this was handled in the work of Colding-De Lellis is by constructing smooth replacements for them and appealing to the replacement theory again.  Here one uses that the minimizing sequence for the restricted $1/j$-problem that was used to produce $V_j$, is actually by a Squeezing Lemma minimizing among {\it all} isotopies in a small enough ball.  Then one can use the fact that a minimizing sequence for area (in the sense of Meeks-Simon-Yau \cite{msy}) has a smooth limit.  In our case what is needed is to prove that given a small ball $B$ centered around $\mathcal{S}$ of isotropy $\mathbb{Z}_n$ and an equivariant surface in the ball with boundary $\{\gamma_i\}_{i=1}^k\subset B$, one can minimize area among $G$-equivariant competitors to produce a smooth minimal surface with boundary $\{\gamma_i\}_{i=1}^k$.  One can even assume the genus of $\Sigma$ is zero.

If $k=1$ and one is thus seeking an equivariant area-minimizing disk, then it follows from the work of Meeks-Yau \cite{my} that {\it any} minimizing disk is in fact equivariant.   Since by Almgren-Simon \cite{almgrensimon} such a minimizing disk is smooth, this would complete the proof.  If $k>1$ and the curves bound multiple planar domains, the issue of whether the minimizing surface is equivariant seems more delicate. The Meeks-Yau cut-and-paste argument to prove that minimizers are equivariant does not appear to work (see the example of annuli on page 227 in \cite{my}).  There is a symmetrization procedure of Lawson \cite{lawson} to construct from a minimizing current an equivariant current with the same area but genus may incease in this process and so it is not directly applicable to our setting.  An elementary example going back to Federer (Section 5.4.17 in \cite{F}) in one lower dimension illustrates that the question of whether equivariant minimizers are minimizers among all competitors is quite delicate: consider four points at the vertices of a square.  Such a configuration is $\mathbb{Z}_4$-equivariant but the minimizing one dimensional current with such a boundary consists of the union or two vertical or horizontal line segments, neither of which is $\mathbb{Z}_4$-equivariant.

Thus to prove regularity of equivariant minimizers in this setting (Proposition \ref{minimizing}), we adapt ideas of Almgren-Simon \cite{almgrensimon} to perform appropriate neckpinches on the equivariantly minimizing sequence and do a replacement procedure so that the sequence consists of disks in a small ball about the singular axis.  This allows us to reduce to the case of disks where the equivariantly minimizing limits and limits minimizing among all isotopies coincide.  If $\mathbb{Z}_n\neq\mathbb{Z}_2$ we give a second argument based on ruling out the various singularities that can occur at $\mathcal{S}$.  The case $n=2$ is special because as observed earlier, the minimizing sequence may contain segments of the singular axis.  
\\
\indent In summary we need the following three ingredients:
\begin{enumerate}
\item $G$-stationary $\Rightarrow$ stationary
\item $G$-stable $\Rightarrow$ stable
\item A minimizing sequence for the equivariant Plateau problem for planar domains has a regular limit.
\end{enumerate}

\section{$G$-equivariant surfaces}
We first make some definitions.
\begin{definition}
\normalfont A varifold $\Sigma$ in $\Sigma\subset M$ is {\it$G$-equivariant} if for all $g\in G$, $g_\#\Sigma=\Sigma$. 
Likewise, a sweepout $\{\Sigma_t\}$ of $M$ is {\it $G$-equivariant} if for all $t$ and $g\in G$, $g_\#\Sigma_t=\Sigma_t$.
\end{definition}

\begin{definition}\label{eqisotopy}\normalfont
An isotopy $\Phi(t)$ is {\it $G$-equivariant} if $\Phi(t)=g^{-1}\circ\Phi(t)\circ g$ for all $t$ and $g\in G$.  Likewise, a vector field $\chi$ is $G$-equivariant if $g^{\#}(\chi)=\chi$ for all $g\in G$.
\end{definition}\noindent

Of course, any $G$-equivariant vector field induces a $G$-equivariant isotopy through integration.  Also note that the set of $G$-equivariant vector fields is itself a vector space: if $X$ and $Y$ are $G$-equivariant vector fields, so is $c_1X+c_2Y$ for any $c_1,c_2\in\mathbb{R}$.

Let us first address whether and how a smooth $G$-equivariant surface $\Gamma$ can intersect the singular set of the group action. At such a point $x\in\Gamma\cap\tilde{\mathcal{S}}$, the tangent plane $T_x\Sigma$ would have to be $G$-equivariant.  Thus we are interested in subgroups $G$ of $SO(3)$ which fix some plane $\mathcal{P}\subset T_xM$ setwise i.e. $g\mathcal{P}=\mathcal{P}$ for all $g\in G$.  In the following, we take the cyclic groups and dihedral groups to include rotations about the $z$ axis. 

We have the following elementary lemma:

\begin{lemma}\label{actions}
The finite subgroups of $SO(3)$ are $\mathbb{Z}_n$, $\mathbb{D}_n$ and the four groups associated with the Platonic solids. None of the Platonic groups fix a plane. The groups $\mathbb{Z}_n$ and $\mathbb{D}_n$ for $n\neq 2$ leave invariant only the $xy$-plane.   The group $\mathbb{Z}_2$ fixes the $xy$ plane and any plane containing the $z$-axis.  The group $\mathbb{D}_2$ fixes the $xy$-plane, the $yz$-plane, and $xz$-plane.
\end{lemma}
\noindent
From Lemma \ref{actions} we conclude immediately that a smooth $G$-equivariant surface $\Gamma$ can only intersect $\tilde{\mathcal{S}}$ in a point of type $\mathbb{Z}_n$ or $\mathbb{D}_n$.  The next lemma addresses points of $\mathbb{Z}_n$ type: it says that it can never be tangent to $\tilde{\mathcal{S}}$ unless $G=\mathbb{Z}_2$ where $\Gamma$ could be a surface Schwarz-reflected through $\mathcal{S}$.

\begin{lemma}\label{intersect}  Let $B$ be a $\mathbb{Z}_n$-ball and $\Sigma$ a smooth embedded $\mathbb{Z}_n$-equivariant surface  contained in $B$ with $\partial\Sigma\subset\partial B$.  Denote by $\mathcal{S}$ the singular set of the action of $\mathbb{Z}_n$ on $B$.
\begin{enumerate}
 \item If $n\neq 2$, then for any $x\subset\Sigma\cap\mathcal{S}$ we have $T_x\Sigma\perp\mathcal{S}$.
\item If $n=2$ then either $\mathcal{S}\subset\Sigma$ (and $\Sigma$ is a Schwarz reflection through $\mathcal{S}$) or else $\Sigma\cap\mathcal{S}$ is empty or a finite set, and for any $x\subset\Sigma\cap\mathcal{S}$ we have $T_x\Sigma\perp\mathcal{S}$.
\end{enumerate}
\end{lemma}

\begin{proof}
We first prove (1). If $x\subset\Sigma\cap\mathcal{S}$, then since $\Sigma$ is smooth it has a unique tangent plane at $x$ in $T_xM$.  But the plane must be $\mathbb{Z}_n$-equivariant,  and the only such plane is the unique plane orthogonal to $\mathcal{S}$ at $x$.  For (2), set $G=\mathbb{Z}_2$ and suppose $x\in\mathcal{S}\cap\Sigma$.  By the reasoning above, the tangent plane $T_x\Sigma$ is either tangent to $\mathcal{S}$ or othogonal.  Suppose $T_x\Sigma$ is tangent to $\mathcal{S}$.  We claim $\mathcal{S}\subset\Sigma$.  By equivariance this then implies that $\Sigma$ is Schwarz reflection through $\mathcal{S}$.  First consider the set 
\begin{equation} 
C=\{x\in\mathcal{S}\;|\; x\subset\Sigma \mbox{ and } \Sigma \mbox{ tangent to } \mathcal{S} \mbox{ at } x\}.
\end{equation}

Certainly $C$ is closed by continuity.  It is nonempty by hypothesis.  We will show it is open, and therefore $C=\mathcal{S}$.  To see this, fix $y\in C$. Since $\Sigma$ is smooth, it can be written as a graph over the tangent plane $T_y\Sigma$ in a small neighborhood (after pulling back via exponential coordinates).  Precisely, let $\exp_y:T_xM\rightarrow M$ be the exponential map and consider the surface 
\begin{equation}
\tilde{\Sigma}=\exp^{-1}_y(\Sigma\cap B_\epsilon(x))
\end{equation}
for suitably small $\epsilon$.  Rotate the coordinates of $T_y M$ so that $T_y\Sigma$ is the $xy$ plane and $\mathcal{S}$ maps to the $x$ axis.  Note that $\mathcal{S}$ maps to an axis because it is a geodesic.  Then $\tilde{\Sigma} = \text{graph}(f)$ where $f(0,0)=0$ by assumption and by equivariance $f(-x,y)=-f(x,y)$ for all $(x,y)$ small enough.   But this implies $f(0,y)=0$ for $y$ small enough.  This is precisely saying that $\Sigma$ contains the geodesic segment $\mathcal{S}$ in a neighborhood of $y$.   Hence $C$ is open and therefore $\mathcal{S}\subset\Sigma$.  
\end{proof}
It remains to see whether and how $\Gamma$ can intersect points of $\tilde{\mathcal{S}}_0$ of type $\mathbb{D}_n$.  

\begin{lemma}\label{intersect2}  Let $B$ be a $\mathbb{D}_n$-ball and $\Sigma$ a smooth embedded $\mathbb{D}_n$-equivariant minimal surface contained in $B$ with $\partial\Sigma\subset\partial B$.   Let $\mathcal{S}$ denote the singular set of $\mathbb{D}_n$ acting on $B$.  It consists of a central point $z$ of isotropy $\mathbb{D}_n$, $\mathcal{S}'$ the union of $2n$ line segments each with both endpoints in $\partial B$ and all meeting at $z$ (i.e., the rotations of isotropy $\mathbb{Z}_2$), and an arc $\mathcal{S}''$ through $z$ orthogonal to $\mathcal{S}'$ with isotropy $\mathbb{Z}_n$.  Note that $\mathcal{S}'$ and $\mathcal{S}''$ intersect only at $z$.  Suppose $\Sigma$ contains $z$ in its support.   Then 
\begin{enumerate}
 \item If $n\neq 2$, then $\Sigma$ contains $\mathcal{S}'$ and is orthogonal to $\mathcal{S}''$.
\item If $n=2$ then either $\Sigma$ contains $\mathcal{S}'$ and is orthogonal to $\mathcal{S}''$, or else $\Sigma$ contains one of the geodesic segments comprising $\mathcal{S}'$ as well as $\mathcal{S}''$ (and is a Schwarz-reflection through this latter set).
\end{enumerate}

\end{lemma}
\begin{proof}
First let us suppose $n\neq 2$ and that $\Gamma$ passes through the point $z$ in the $\mathbb{D}_n$ ball.  Consider as in the proof of Lemma \ref{intersect} $\tilde{\Gamma}=\exp_z^{-1}(\Sigma\cap B(z))$.  After a rotation, by equivariance Lemma \ref{actions}, $T_{(0,0,0)}\tilde{\Gamma}$ must be the $xy$ plane and near $(0,0,0)$, $\tilde{\Gamma}$ is a graph $G$ over its tangent plane.  It follows by the equivariance exactly as in Lemma \ref{intersect} that $\Gamma$ vanishes on the arcs comprising $\mathcal{S}'$ and thus $\Sigma$ contains $\mathcal{S}'$.  If $n=2$, then by Lemma \ref{actions} there are three possible configurations for the tangent plane at $z$ which gives (2).
\end{proof}

Finally, we see:
\begin{lemma}\label{intersect3}
The number of intersections of a smooth $G$-equivariant surface intersecting $\mathcal{S}$ transversally with each arc of constant isotropy is unchanged after applying an equivariant isotopy.  
\end{lemma}
\begin{proof}
Equivariant isotopies, as isotopies, preserve the set of smooth surfaces.  Thus let $\Sigma$ be a surface intersecting each arc of constant isotropy in $M$ a fixed number of times.  Let $\phi_t$ (for $t\in[0,1]$) be an equivariant isotopy so that $t_0>0$ is the first time that $\phi_{t_0}(\Sigma)$ intersects an arc of constant isotropy $\mathbb{Z}_n$ in a different number of points or becomes tangent to an arc of $\mathbb{Z}_n$ isotropy.  If the surface $\phi_{t_0}(\Sigma)$ for instance is tangent to $\mathbb{Z}_n$ at a point $p$, then by Lemma \ref{intersect} since $\phi_{t_0}(\Sigma)$ is smooth, it follows that $n=2$. For $t$ slightly less than $t_0$ $\phi_t(\Sigma)$ consists near $p$ of an even number of graphs each converging to $\phi_{t_0}(\Sigma)$ near $p$. To see this, if any graph were preserved by the $\mathbb{Z}_2$ action it would vanish on the singular axis by equivariance, in which case $\phi_{t_0}(\Sigma)$ would contain several sheets passing through the axis and thus not be smooth.  Thus the graphs are all interchanged by the group $\mathbb{Z}_2$, which means the number of them is even.  It follows that $\phi_{t_0}(\Sigma)$ is achieved as a limit with multiplicity and so $t\rightarrow\phi_t(\Sigma)$ are not a smoothly varying family of surfaces for $t$ near $t_0$.  This contradicts the fact that $\phi$ is an isotopy. 

Thus we need only consider the case that $\Sigma$ intersects an arc of constant isotropy orthogonally $k$ times for $t\leq t_0$, and yet $\phi_{t_0}(\Sigma)$ contains fewer or more than $k$ intersection points with the arc.  But it follows from Lemma  \ref{intersect} that the intersections of $\phi_{t}(\Sigma)$ with the singular arcs are all orthogonal.  Thus since the surfaces $\phi_{t}(\Sigma)$ vary smoothly, $\phi_{t}(\Sigma)$ around each point of intersection with the singular arc is a graph intersecting the singular arc once, and thus the number of intersection points of $\phi_{t_0}(\Sigma)$ with the singular arc is constant for $t$ near $t_0$.  This contradicts that $t_0$ is the first time the number changes. 
\end{proof}
\subsection{Existence of a $G$-stationary varifold}

\begin{definition}\normalfont
A varifold $\mathcal{V}$ is called {\it $G$-stationary} in an open set $\mathcal{U}$ if for every $G$-equivariant vector field $\chi$ supported compactly in $\mathcal{U}$, we have $\delta\mathcal{V}(\chi)=0$. \end{definition} 

The next key lemma says that a varifold that has zero variation with respect to equivariant deformations is in fact stationary with respect to all deformations.  It is a weak formulation of the Symmetric Critical Point principle articulated by Palais \cite{palais}.

\begin{lemma}\label{gstat}
A $G$-equivariant $G$-stationary varifold $\mathcal{V}$ in $M$ is stationary.
\end{lemma}
\begin{proof}
 Given any vector field $\chi$ on $M$ we must show $\delta\mathcal{V}(\chi)$ =0.   To do this, we will construct from $\chi$ a $G$-equivariant vector field $\chi_G$ such that $\delta\mathcal{V}(\chi_G)=\delta\mathcal{V}(\chi)$.  Since $\delta\mathcal{V}(\chi_G)$ vanishes by $G$-stationarity, we will be done.   Let $\Psi(t)$ be the one parameter family of diffeomorphisms that generates $\chi$. 

For each $g\in G$, define a new one parameter family of diffeomorphisms: 
\begin{equation}\label{areaequal3}
\Psi_g(t)=g^{-1}\circ\Psi(t)\circ g. 
\end{equation}  By construction $\Psi_g(0)$ is the identity for each $g\in G$.  By equivariance of $\mathcal{V}$ and \eqref{areaequal3} we have that for all $t$ and $g\in G$, 
\begin{equation}\label{areaequal2}
g_{\#}\circ\Psi_g(t)_\#(\mathcal{V})=\Psi(t)_\#\circ g_{\#}(\mathcal{V})=\Psi(t)_\#(\mathcal{V})
\end{equation}  
Taking the mass of both sides on \eqref{areaequal2} and using the fact that $G$ acts by isometries we thus obtain
\begin{equation}\label{areaequal}
||\Psi_g(t)_\#(\mathcal{V})||=||\Psi(t)_\#(\mathcal{V})||.
\end{equation} 

 Denote by $\chi_g$ the vector field generated by the one-parameter family $\Psi_g(t)$.  It follows from \eqref{areaequal} that 
\begin{equation}\label{equal}
\delta\mathcal{V}(\chi_g)=\delta\mathcal{V}(\chi).
\end{equation}
Finally let $\chi_G$ be the $G$-equivariant vector field given as: $$\chi_G=\frac{1}{|G|}\sum_{g\in G} \chi_g.$$ 
To see that $\chi_G$ is equivariant, observe first that $$h^{\#}\chi_g(hx)=\frac{d}{dt}\Bigr|_{t=0} h\circ g^{-1}\circ\chi(t)\circ g(x) =\chi_{gh^{-1}}(hx)$$ so that $$h^{\#}\chi_G(hx)=\frac{1}{|G|}\sum_{g\in G} \chi_{gh^{-1}}(hx)=\frac{1}{|G|}\sum_{g\in G} \chi_g(hx)=\chi_G(hx),$$ where the middle equality follows since the elements $\{gh^{-1}\;|\; g\in G\}$ give a reordering of the sum.  

By linearity, \eqref{equal}, and the fact that $\chi_G$ is equivariant, we obtain $$0=\delta\mathcal{V}(\chi_G)=\frac{1}{|G|}\sum_{g\in G}\delta\mathcal{V}(\chi_g)=\delta\mathcal{V}(\chi).$$
\end{proof}
Now we can state the main result of this section which follows directly from the arguments of \cite{cd}.  It says that we can ``pull-tight" a sweepout so that at least all the min-max sequences have stationary limits.  
\begin{proposition}
There exists a minimizing sequence of $G$-sweepouts of $M$ so that any min-max sequence obtained from it converges to a stationary varifold.
\end{proposition}
\begin{proof}
The proof of Proposition 4.1 in \cite{cd} is a formal argument that extends with trivial modifications to show that a minimizing sequence can be chosen so that all min-max sequences converge to a $G$-stationary varifold: In their notation, one replaces the set $\mathcal{V}_\infty$ with the set $\mathcal{V}^G_\infty$ of $G$-stationary varifolds.  In constructing the map $H_V$ via a partition of unity in their Step 1), one needs only that a vector field constructed via a sum of $G$-equivariant vector fields is itself $G$-equivariant, which follows directly from the definitions.   Finally Lemma \ref{gstat} implies that all $G$-stationary limits are stationary varifolds.  
\end{proof}

\section{Regularity at the Singular Locus $\mathcal{S}$}
In the original theory of Simon-Smith \cite{smith} and Pitts \cite{pitts}, the regularity of min-max limits derives from the fact that they are well-approximated by stable surfaces which satisfy \emph{a priori} curvature bounds.  Namely, one has

\begin{proposition}\label{schoene}(Schoen's curvature estimates \cite{schoen})
A sequence of stable minimal surfaces $\Sigma_j$ in $U$ with $\partial\Sigma_j\subset\partial U$ has a convergent subsequence.
\end{proposition}

In the equivariant setting we will see that min-max sequences are approximated by $G$-stable surfaces (i.e., surfaces that are stable among $G$-isotopies).  First we introduce the notion of $G$-stability and then show that it is equivalent to stability for surfaces intersecting the singular set orthogonally and thus we can still make use of Proposition \ref{schoene}.

\subsection{$G$-stability}
\begin{definition}\normalfont
Let  $\Sigma$ be a smooth $G$-equivariant surface contained in a $G$-ball.  Choose a normal vector field $n$ on $\Sigma$.  Let us call a smooth function $\phi$ defined on $\Sigma$ an {\it equivariant deformation} if for all $t$ small enough, the following set is $G$-equivariant:
\begin{equation}\label{eqdef}
\Sigma_{t\phi} = \{\exp_p(n(p)t\phi(p)) \;|\; p\in \Sigma\}.
\end{equation}
In other words, $\phi$ is an equivariant deformation if moving normally to $\Sigma$ according to $\phi$ gives rise to $G$-equivariant surfaces.  Let us denote by $C^\infty_G(\Sigma)$ the space of smooth equivariant deformations of $\Sigma$ that vanish on $\partial\Sigma$.
\end{definition}

\begin{lemma}\label{normalrev}
Let $G$ be either $\mathbb{Z}_n$, $\mathbb{D}_n$ or one of the three Platonic groups.  Suppose $\Sigma$ is a $G$-equivariant surface (potentially disconnected) contained in a $G$-ball so that $G$ acts freely on $\partial\Sigma$.  Then there is a choice of $G$-equivariant normal vector field $n$ on $\Sigma$.  After making this choice there is a canonical identification 
\begin{equation}\label{n}
C^\infty_{G}(\Sigma)=\{f\in C_c^\infty(\Sigma) |\;f(gx)=f(x) \mbox{ for all } g\in G \mbox{ and } x\in\Sigma\}.
\end{equation}
\end{lemma}

\begin{remark}\label{ddd}\normalfont
The assumption in Lemma \ref{normalrev} that the group $\mathbb{Z}_n$ acts freely on $\partial\Sigma$ is necessary.  Consider for example the unit $3$-ball $B$ in $\mathbb{R}^3$.  Let $\Sigma$ be the disk $B\cap\{z=0\}$, and consider the group $\mathbb{Z}_2$ consisting of the identity and the $180^o$ rotation about the $x$-axis.
\end{remark}

\begin{proof}
Since $G$ consists of isometries preserving $\Sigma$, it follows that for all $g\in G$ and $p\in\Sigma$, and any normal vector $n(p)$, the vector $g_{\#}n(p)$ is equal to either $n(gp)$ or $-n(gp)$.  Suppose for the moment that $g_{\#} n(p) =n(gp)$ for all $p\in\Sigma$ and $g\in G$.  Given $\phi\in C^\infty(\Sigma)$, we clearly have
\begin{equation}\label{ww}
g\exp_p(tn(p)\phi(p))=\exp_{gp}(g_{\#}n(p)t\phi(p)).
\end{equation}
If in addition $\phi\in C^\infty_G(\Sigma)$ then \eqref{ww} along with the assumption about the normal vector under $G$ imply:
 \begin{equation}\label{www}
g\exp_p(tn(p)\phi(p))=\exp_{gp}(n(gp)t\phi(gp)).
\end{equation}

But \eqref{www} after relabeling the right hand side implies $\Sigma_{t\phi}$ is equivariant.  The equalities above are reversible so that we obtain that those variations specified in \eqref{n} are in fact the only equivariant deformations.

It remains to show one can choose a normal vector consistently so that $g_{\#}n(p)=n(gp)$ for all $g\in G$, and $p\in\Sigma$.  For simplicity assume $G=\mathbb{Z}_n$. Let $\Sigma'$ be a component of $\Sigma$. The surface $\Sigma'$ divides $B$ into two connected components, $B_1$ and $B_2$.  If $\Sigma'$ intersects the singular set $\mathcal{S}$ of the group action, it follows from Lemma \ref{intersect} that the intersection is orthogonal to $\mathcal{S}$ (since $G$ acts freely on $\partial\Sigma$ the first case in (2) of Lemma \ref{intersect} cannot occur).  In this case, if $p\in\Sigma\cap\mathcal{S}$ then $p$ is preserved by the group action and thus a generator of $\mathbb{Z}_n$ acts on the tangent space in $B$ at $p$ by fixing the orthogonal direction to $T_p\Sigma'$ and rotating  $T_p\Sigma'$ by $2\pi/n$ about $\mathcal{S}$.  It follows that $B_1$ and $B_2$ are preserved under $G$ (and hence $\Sigma'$ is), from which we see that $G$ cannot flip the components $B_1$ and $B_2$, i.e. there is a well-defined $G$-equivariant normal vector field.   

If instead $\Sigma'$ is disjoint from the singular set $\mathcal{S}$, then the singular set is entirely contained in one of the components $B_1$ or $B_2$, say $B_1$ (again this is true because $G$ acts freely on $\partial\Sigma$).  Choose the normal vector field on $\Sigma'$ to point into $B_1$.  There are two cases: either i) $g(\Sigma')=\Sigma'$ or ii) $g(\Sigma')$ is some other component of $\Sigma$, $\Sigma''$.  Since the singular set $\mathcal{S}$ is preserved by $G$ pointwise, given any element $g\in G$, the open set $g(B_1)$ still contains $\mathcal{S}$.  Thus in case i) $g(B_1)=B_1$ and $g(B_2)=g(B_2)$ and thus the given vector field is equivariant.  In case ii), the minimal set equivariant under $G$ containing $\Sigma'$ consists of several copies of $\Sigma$: $\Sigma$, $g(\Sigma)$, ... $g^{j}(\Sigma)$.  Each such component divides $B$ into two pieces, one of which contains $\mathcal{S}$.  For each $g^k(\Sigma')$ choose the normal vector field to point into the component of $B\setminus g^k(\Sigma')$ containing $\mathcal{S}$.  Since $G$ preserves $\mathcal{S}$, this gives a well-defined equivariant normal vector field.

The cases $G=\mathbb{D}_n$ and three polyhedral groups follow analagously.

\end{proof}
In the remainder of this section, let $G$ be either $\mathbb{Z}_n$, $\mathbb{D}_n$ or one of the three Platonic groups. 

\begin{definition}\normalfont
A smooth surface $\Sigma$ (potentially with boundary) contained in a $G$-set is {\it $G$-stable} if for every $\phi\in C^\infty_G(\Sigma)$, one has the stability inequality: 
\begin{equation}
-\int_\Sigma \phi L_\Sigma\phi\geq 0,
\end{equation}
where $L_\Sigma=\Delta_\Sigma+|A|^2+Ric(n,n)$.
\end{definition}

The point is that if we are assuming the surfaces intersect the singular set transversally, $G$-stability is equivalent to stability:
\begin{proposition}\label{gstability}
Let $\Sigma$ be a smooth $G$-stable equivariant minimal surface (potentially disconnected) contained in a $G$-ball such that $G$ acts freely on $\partial\Sigma$.  Let $\phi_1\geq 0$ be the lowest eigenfunction of the stability operator $L$ (vanishing at $\partial\Sigma$). Then 
\begin{enumerate}
\item $\phi_1$ is $G$-equivariant (i.e. $\phi_1(gx)=\phi_1(x)$ for all $g\in G$)
\item $\Sigma$ is stable.
\end{enumerate}
\end{proposition}
\begin{proof}
For (1) observe that for any $g\in G$, since $G$ acts by isometries, we have by the characterization of eigenfunctions in terms of Rayleigh quotient that the function $\phi_2(x)=\phi_1(g(x))$ is also an eigenfunction of the stability operator with the same eigenvalue as $\phi_1$.  But the dimension of the eigenspace of the lowest eigenfunction  is one dimensional, so that $\phi_2=c\phi_1$ for some $c\in\mathbb{R}$.  Iterating we obtain $\phi_1(g^nx)=c^n\phi_1(x)=\phi_1(x)$ implying $c\in\{-1,1\}$.   Since $\phi_2(x)=\phi_1(gx)\geq 0$, it follows that $c=1$, and (1) is proved.


For (2), consider the two Rayleigh quotients:
\begin{equation}\label{imp}
\lambda_1=\inf_{f\in C_c^\infty(\Sigma)}\frac{-\int_\Sigma fLf}{\int_\Sigma f^2}
\end{equation}
and 
\begin{equation}\label{imp2}
\lambda^G_1=\inf_{f\in C_G^\infty(\Sigma)}\frac{-\int_\Sigma fLf}{\int_\Sigma f^2}.
\end{equation}

By $G$-stability, it follows that $\lambda^G_1\geq 0$. We claim 
\begin{equation}\label{oy}
\lambda_1=\lambda^G_1 
\end{equation} 
and thus $\lambda_1\geq 0$.   Recall that by Lemma \ref{normalrev}, we can identify $C_G^\infty(\Sigma)$ with the functions $f$ on $\Sigma$ satisfying $f(gx)=f(x)$ for all $x\in\Sigma$ and $g\in G$.  To see \eqref{oy} observe that since the infimum in \eqref{imp} is taken over a larger set than in \eqref{imp2}, it follows that $\lambda_1\leq\lambda^G_1$.  For the opposite inequality, by (1), the eigenfunction $\phi_1$ attaining the infimum in \eqref{imp} is in fact equivariant, and thus contained in $C_G^\infty(\Sigma)$.  Thus $\lambda_1\geq\lambda^G_1$.  \end{proof}
\begin{remark}\label{whynoref}\normalfont
For general groups $G$, $G$-stability need not be equivalent to stability. Indeed, denote by $\tau$ the involution in the round three sphere through an equator, $X$.  The only normal deformation of $X$ that is equivariant under the group $(1,\tau)$ is the zero deformation.  Thus the equator is trivially $(1,\tau)$-stable.  However, the equator is not stable among all deformations.  This is one reason why we restrict to groups $G$ so that $M/G$ has no boundary.   Similarly, if a minimal surface contains an arc of the singular set as in Remark \ref{ddd}, $\mathbb{Z}_2$-stability need not be equivalent to stability.
\end{remark}

\subsection{Existence of a $G$-almost minimizing min-max sequence}
We will show that one can always choose a min-max sequence that has the {\it almost-minimizing} property relative to $G$-equivariant isotopies. This property was first introduced by Almgren \cite{A} in the 60s and was used by Pitts \cite{pitts} in his thesis to prove regularity of min-max limits.  To formulate the property, it is most convenient to work on $M/G$, the space of orbits.   Of course $M/G$ is a singular space but we can still endow it with a distance function: For $[x],[y]\in M/G$, set:
\begin{equation}\label{distance}
d_{M/G}([x],[y])=\inf_{g\in G, h\in G}d_M(gx,hy).
\end{equation}
The following is the essential point for getting the a.m. property:
\begin{lemma}
Endowed with the distance function \eqref{distance}, $M/G$ is a metric space.  
\end{lemma}
\begin{proof}
Symmetry is obvious.  For reflexivity, if $d_{M/G}([x],[y])=0$, we have $d_M(gx,hy)=0$ for some $g$ and $h$, so that $gx=hy$, or $x=g^{-1}hy$ so that $[x]=[y]$.  For the triangle inequality, consider $[x],[y],[z]\in M/G$.  Let $g$ and $h$ be such that $d_{M/G}([x],[y])=d_M(gx,hy)$ and let $g'$ and $h'$ be such that $d_{M/G}([y],[z])=d_M(h'y,g'z)$.   Then 
$$d_{M/G}([x],[z])\leq d_M(gx,(h'h^{-1})^{-1}g'z)\leq d_M(gx,hy)+d_M(hy,(h'h^{-1})^{-1}g'z).$$
  But  because $h'h^{-1}$ is an isometry we have $$d_M(hy,(h'h^{-1})^{-1}g'z)=d_M(h'y,g'z),$$ which gives precisely the triangle inequality.
\end{proof}
We now introduce the relevant $G$-invariant objects of study:
\begin{definition} \normalfont
We call an open set $\mathcal{U}\subset M$ a \emph{$G$-set} if $g \mathcal{U}=\mathcal{U}$ for all $g\in G$. 
Given a $G$-set $\mathcal{U}$ and $G$-invariant surface $\Sigma\subset\mathcal{U}$, we say $\Sigma$ is \emph{$(G,\delta,\epsilon)$-almost minimizing in $\mathcal{U}$} if there is no $G$-equivariant isotopy $\psi_t:\mathcal{U}\rightarrow\mathcal{U}$ so that both
\begin{enumerate}
\item $|\psi_1(\Sigma)|\leq |\Sigma|-\epsilon$ 
\item $|\psi_t(\Sigma)|\leq |\Sigma|+\delta$ for $0\leq t\leq 1$.
\end{enumerate}
\end{definition}

A surface is \emph{$(G,\epsilon)$-almost minimizing in $\mathcal{U}$} if it is $(G,\epsilon/8,\epsilon)$-almost minimizing.  Given a pair $(\mathcal{O}^1,\mathcal{O}^2)$ of open $G$-sets in $M$ we say that a surface $\Sigma\subset M$ is \emph{$(G,\epsilon)$-almost minimizing in $(\mathcal{O}^1,\mathcal{O}^2)$} if it is $(G,\epsilon)$-almost minimizing in at least one of $\mathcal{O}^1$ or $\mathcal{O}^2$.  Denote by $\mathcal{CO}_G$ the set of pairs $(\mathcal{O}^1,\mathcal{O}^2)$ of $G$-sets so that 
\begin{equation}
d_{M/G}(\pi(\mathcal{O}^1),\pi(\mathcal{O}^2))\geq 2\min(\text{diam}_{M/G}\pi(\mathcal{O}^1),\text{diam}_{M/G}\pi(\mathcal{O}^2)).
\end{equation}  

\noindent As in \cite{cd} we have the lemma (that only uses the metric space property of $M/G$):
\begin{lemma}\label{lemma}
If $(\mathcal{O}^1,\mathcal{O}^2)\in\mathcal{CO}_G$ and $(\mathcal{U}^1,\mathcal{U}^2)\in\mathcal{CO}_G$, then there are $i,j\in\{1,2\}$ so that $d_{M/G}(\pi(\mathcal{O}^i),\pi(\mathcal{U}^j))>0$ (and thus by \eqref{distance} we also obtain $d_{M}(\mathcal{O}^i,\mathcal{U}^j)>0$).
\end{lemma}

Because of Lemma \ref{lemma} the proof of Proposition 5.1 in \cite{cd} carries over identically to imply:

\begin{lemma}
There is a min-max sequence $\Sigma_L$ so that $\Sigma_L$ converging to a stationary varifold so that $\Sigma_L$ is $1/L$ almost minimizing in every $(U^1,U^2)\in\mathcal{CO}_G$.
\end{lemma}

We now define a $G$-equivariant function $\tilde{r}:M\rightarrow\mathbb{R}^+$ as follows.  Given $x\in (M/G)\setminus\mathcal{S}$, set $r(x)=\frac{1}{2}\text{dist}_{M/G}(x,\mathcal{S})$.  For $x\in\mathcal{S}$, choose $r(x)$ so small so that $r(x)$ intersects the singular set $\mathcal{S}$ in a fixed number of geodesic segments passing through $x$.  If $x$ is contained in the singular set with isotropy $\mathbb{Z}_n$ for instance, then $B_{r(x)}(x)$ intersects the singular set twice.  Then set $\tilde{r}(x)=r(\pi(x))$, giving a $G$-equivariant function.
Finally denote by $\mathcal{AN}^G_{\tau}(x)$ the collection of lifts to $M$ of annuli of outer radius at most $\tau$ about $x\in M/G$.  Potentially shrinking $\tilde{r}(x)$ we obtain as in Proposition 5.1 in \cite{cd} directly:

\begin{proposition}\label{am}
There exists a $G$-equivariant function $\tilde{r}:M\rightarrow\mathbb{R}^+$ and a min-max sequence $\Sigma_j$ so that:
\begin{enumerate}
\item For $x\notin\tilde{\mathcal{S}}$, $\tilde{r}(x)<\mbox{dist}(x,\tilde{\mathcal{S}})$
\item For $x\in\tilde{\mathcal{S}}_1$, $\tilde{r}(x)<\mbox{dist}(x,\tilde{\mathcal{S}_0})$
\item  The sequence $\Sigma_j$ is $(G,1/j)$-almost-minimizing in every $An\in\mathcal{AN}^G_{\tilde{r}(x)}(x)$ for all $x\in M/G$.
\item In any such $An$ from (2), $\Sigma_j$ is disjoint from $\tilde{\mathcal{S}}$ for $j$ large enough.
\item In any such $An$ from (2), $\Sigma_j$ has genus $0$.
\item $\Sigma_j$ converges to a stationary varifold $\Sigma_\infty$.
\end{enumerate}

\end{proposition}
\begin{proof}
This follows directly as in the Appendix in \cite{CGK}.  For the proof of (4) and (5), by Lemma \ref{intersect3} the number of points of intersection of $\Sigma_j$ with $\tilde{\mathcal{S}}$ is independent of $j$.  After passing to a subsequence we can let $\mathcal{P}$ be the set of limits of these points.   The genus of a sequence of surfaces of bounded genus can collapse into at most finitely many points, $\mathcal{G}$ (Lemma I.0.14 in \cite{CM}).  Shrink $\tilde{r}(x)$ appropriately so that the annuli of outer radii at most $\tilde{r}(x)$ are disjoint from $\mathcal{P}\cup\mathcal{G}$. 
\end{proof}

For any $x\in (M/G)\setminus\mathcal{S}$, by Proposition \ref{am}, the sequence $\Sigma_j$ is $(G,1/j)$-almost minimizing in the $|G|$ disjoint components comprising annuli in $\mathcal{AN}_{\tilde{r}}^G(\pi(x))$.  Thus it is $(1/(8|G|j),1/j)$-almost minimizing among {\it all} not necessarily equivariant isotopies in each such disjoint component.  Thus by Theorem 7.1 in \cite{cd} we obtain that $\Sigma_\infty$ is smooth in $M\setminus\tilde{\mathcal{S}}$.  
\\
\\ \indent
The remainder of this section will be taken up with proving the regularity of $\Sigma_\infty$ over the singular set $\tilde{\mathcal{S}}$.

For $x\in\tilde{\mathcal{S}}$, consider the (necessarily fewer than $|G|$) disjoint components of an annulus in $\mathcal{AN}_{\tilde{r}(x)}^G(\pi(x))$.  Then in each such disjoint component $A$, $\Sigma_j$ is  $(1/(8|G|j),1/j)$-almost minimizing among $G$-equivariant isotopies supported in $A$ (which follows from the definition of almost minimizing and the fact that one can concatenate isotopies with disjoint supports).

To show that $\Sigma_\infty$ is smooth over $\tilde{\mathcal{S}}$ we have to construct smooth replacements for it in annuli centered on $\tilde{\mathcal{S}}$.  Precisely, for any $x\in\tilde{\mathcal{S}}$ and component of $An\in\mathcal{AN}^G_r(\pi(x))$, we must produce a stationary varifold $V_\infty$ in $M$ so that the following hold: 
\begin{enumerate}
\item  $||V_\infty||=||\Sigma_\infty||$, 
\item $V_\infty=\Sigma_\infty$ on $M\setminus\overline{An}$
\item $V_\infty$ restricted to the set $An$ is a smooth stable minimal surface.
\item $V_\infty$ intersects $\tilde{\mathcal{S}}$ orthogonally in finitely (potentially zero) points in $An$
\end{enumerate}
Then the regularity at $\tilde{\mathcal{S}}$ follows from Proposition 6.3 in \cite{cd}.  To construct a smooth replacement for $\Sigma_\infty$ in $An$, we first construct from $\Sigma_j$ a smooth $G$-stable minimal surface $V_j$ that agrees with $\Sigma_j$ outside of $An$.  To do this, let $\mathcal{I}^G_j$ denote the set of $G$-isotopies $\psi_t$ (for $0\leq t\leq 1$) supported in $An$ so that
\begin{equation}\label{cons}
|\psi_t(\Sigma_j)|\leq |\Sigma_j|+1/(8j|G|) \mbox{    for } 0\leq t\leq 1.
\end{equation}
Let 
\begin{equation}
m_j = \inf_{\phi\in\mathcal{I}^G_j}|\phi_1(\Sigma_j)|.
\end{equation}

For each $j$, let $\phi^k$ be a minimizing sequence is isotopies in $\mathcal{I}^G_j$ so that $|\phi_1^k(\Sigma_j)|\rightarrow m_j$.  Denote by $V_j$ the varifold limit of $\phi^k_1(\Sigma_j)$ as $k$ tends to infinity.   The surface $V_j$ is $G$-stable in $An$.  Assume for the moment that $V_j$ is in addition a smooth surface inside $An$.

Taking $j$ large enough, $\Sigma_j$ is disjoint from $\tilde{\mathcal{S}}$ in $An$ by item (4) in Proposition \ref{am}.  Since the curves comprising $\Sigma_j\cap\partial (An)$ are disjoint from $\tilde{\mathcal{S}}$, the group $G$ acts freely on them.  The boundary of $An$ has an outer component $O$ and an inner component $I$.  Because of the convexity of the annulus near $O$, it follows from the boundary regularity (Lemma 8.1) proved in \cite{dep} that $V_j\cap\partial O = \Sigma_j\cap\partial O$, and thus $G$ acts freely on $V_j\cap\partial O$ as well.  Consider $An_j$ a family of subannuli of $An$ converging to $An$ with outer radius equal to that of $An$ and inner radius slightly less.  Since $V_j$ is smooth in $An$, it follows that $G$ acts freely on $\partial (An_j)\cap V_j$ as well.  Thus by Proposition \ref{gstability}, $G$-stability implies stability in $An_j$.  Taking a diagonal subsequence in $j$, Schoen's estimates for stable surfaces \cite{schoen} allow one to extract a limit $V_\infty$ from the $V_j$ that satisfies (1), (2), and (3) above in the requirements for a replacement.  Item (4) follows from (3), Lemma \ref{intersect} and again the boundary regularity proved in Lemma 8.1 in \cite{dep}. In the next subsection we prove that the stable replacements $V_j$ are smooth which will complete the proof that $\Sigma_\infty$ is smooth.
\subsection{Regularity of the stable replacements $V_j$}
The surfaces $V_j$ arise as solutions to an area minimization problem with a constraint \eqref{cons}, so their regularity is not surprising.  In the non-equivariant setting, the regularity of the stable replacements $V_j$ is proved (indirectly) in Lemma 7.4 in \cite{cd} by showing that they too have smooth replacements.  

There are two ingredients in the proof.  The first is the Squeezing Lemma (Lemma 7.6 in \cite{cd}).  It gives that when minimizing area to produce $V_j$ from $\Sigma_j$ under the constraint \eqref{cons} that area never goes up too much in the process, on a small enough scale, the minimizing sequence is actually minimizing among {\it all} isotopies.  In other words, on a small enough scale the constraint that area not increase too much disappears.  

The second ingredient required is that in any small enough ball $B$ supported in $An$, minimizers to the equivariant area minimizing problem (without any constraint) using isotopies has a regular limit.   For a ball $B$ supported in $An$ away from $\tilde{\mathcal{S}}$ there is nothing to prove as the result follows from Meeks-Simon-Yau \cite{msy}.  Thus we need only consider the situation when $B$ is a ball centered about $\tilde{\mathcal{S}}$ in $An$ and we must prove that minimizers among equivariant competitors are regular.  We carry this out in Proposition \ref{minimizing}.  The reader may peruse the proof of Lemma 7.6 in \cite{cd} to see that these are the only two necessary changes.

To prove the Squeezing Lemma in the equivariant setting, one need only check that the radial dilation map is itself equivariant.  This is straightforward but for completeness we include the argument.   Indeed, let $B_\rho(x)$ be a ball about $x\in\tilde{\mathcal{S}}$ of radius smaller than the injectivity radius of $M$.  For any $\eta<1$, denote by $I_\eta$ the dilation map defined on $B_\rho(x)$ taking $\exp tV$ to $\exp \eta tV$ for any $V$ lying in the unit tangent sphere in $T_xM$.  

\begin{lemma}
$I_\eta$ is $G$-equivariant.
\end{lemma} 
\begin{proof}
Consider the geodesic path $\gamma(t)=\exp(tV)_{t=0}^{t=\rho}$ for some $V\in T_xM$ of unit length.  We need to show $$g(I_\eta(\gamma(t))) = I_\eta(g(\gamma(t)).$$  But since $g$ is an isometry, the path $g(\gamma(t))$ is another geodesic passing through $x$ at $t=0$.  Thus $g(\gamma(t))=\exp(tW)$ for some $W$ in the unit sphere in $T_xM$.  So 
\begin{equation}
I_\eta(g\gamma(t))=I_\eta(\exp(tW))=\exp(\eta tW),
\end{equation}
 where the last equality is by definition of $I_\eta$.  On the other hand, $g(I_\eta(\gamma(t)))$ is another geodesic passing through the origin, so $g(I\eta(\gamma(t)))=\exp(t\eta W')$ for some $W'$ in the unit sphere in $T_p M$.  We must show $W=W'$. Consider now the differentials: $$d_{x} g(I_\eta(\gamma(t)))=dg\circ dI_\eta(V)$$ and $$d_x(I_\eta(g(\gamma(t))) = dI_\eta\circ dg(V).$$  Since $dI_\eta$ is just scalar multiplication, it commutes with $dg$ and we see that $g(I_\eta(\gamma(t)))$ and  $I_\eta(g(\gamma(t))$ have the same derivatives at $0$, so $W = W'$. \end{proof}


We now provide the final ingredient:

\begin{proposition}\label{minimizing}
Suppose $\mathbb{Z}_n$ acts on a $3$-ball $B$ with singular set a geodesic segment $\mathcal{S}$.  Let $\{\gamma_i\}_{i=1}^k$ be a collection of Jordan curves in $\partial B$ bounding a $\mathbb{Z}_n$-equivariant surface $\Sigma\subset B\setminus\mathcal{S}$ of genus $0$ and so that $\mathbb{Z}_n$ acts freely on the curves $\{\gamma_i\}_{i=1}^k$.  Consider a minimizing sequence $\Sigma_i$ for area among $\mathbb{Z}_n$-isotopies supported in $B$.  Then after passing to a subsequence (not relabeled) $\Sigma_i$ converges with multiplicity $1$ to a smooth embedded $\mathbb{Z}_n$-equivariant minimal surface $V$ with boundary $\{\gamma_i\}_{i=1}^k$ and genus $0$.
\end{proposition}
  
We begin with an elementary lemma:

\begin{lemma}\label{basic}
Let $G=\mathbb{Z}_n$ act on $\mathbb{R}^3$ by rotations about the $z$-axis.  Let $C$ be a non-flat $G$-equivariant stationary cone that is smooth away from the $z$-axis.   The the support of $C$ is the union of several half planes bounded by the $z$-axis.
\end{lemma}
\begin{proof}
The support of $C$ is a cone over a geodesic net in $\mathbb{S}^2$, which is a union of geodesic segments.  If the segments intersect at a point away from the $z$-axis, $C$ could not be smooth away from the $z$-axis.  Thus the only place segments come together is at the north pole and south pole, which forces the desired decomposition.  
\end{proof}

Let us also recall some basic properties of minimizing disks that follow from the Meeks-Yau \cite{my} cut-and-paste arguments:

\begin{proposition} (Theorem 4 in Meeks-Yau \cite{my})
\indent
\begin{enumerate}\label{basic2}
\item Let $B$ be a uniformly convex ball in a three-manifold.  Any two area minimizing disks with disjoint boundaries in $\partial B$ are themselves disjoint.  
\item Suppose $B$ is acted upon by a $\mathbb{Z}_n$ group of isometries and let $\gamma$ be a closed equivariant curve in $\partial B$ acted upon freely.   Then any area minimizing disk bounded by $\gamma$ is itself equivariant. 
\end{enumerate}
\end{proposition}

Let us now prove Proposition \ref{minimizing}:
\begin{proof}
Since $\Sigma_j$  is minimizing among $G$ isotopies, the sequence $\Sigma_j$ converges to a $G$-stationary varifold $\mathcal{V}$. By Lemma \ref{gstat}, $\mathcal{V}$ is stationary.  It remains to prove the regularity.  In any small enough ball supported away from the axis $\mathcal{S}$, the varifold $\mathcal{V}$ has a smooth replacements by Meeks-Simon-Yau \cite{msy}.  Thus one can apply the theory of replacements in \cite{cd}, as well as the boundary regularity proved in \cite{dep} to obtain that:
\begin{enumerate}[label=\roman*)]
\item $\mathcal{V}$ is smooth and embedded in $B\setminus\mathcal{S}$, 
\item $\mathcal{V}$ has boundary in $\partial B$ given by $\{\gamma_i\}_{i=1}^k$,
\item $\mathcal{V}$ is contained in a convex closed set $A\subset B$ that touches $B$ only at the curves $\{\gamma_i\}_{i=1}^k$ in a transversal manner to $\partial B$
\item the total genus of $\mathcal{V}$ is $0$.
\end{enumerate}

The statement ii) follows since $\mathbb{Z}_n$ acts freely on the boundary curves, and thus since regularity is a local statement, one can work in a small neighborhood of any point $p\in\gamma_i$ to see that the tangent cone at such a point is a half disk directly as in Lemma 8.1 in \cite{dep}).  It remains to prove the regularity of $\mathcal{V}$ at the singular set $\mathcal{S}$.

First suppose $n\neq 2$.  Consider $\mathcal{S}'=\text{supp}(\mathcal{V})\cap\mathcal{S}$.  By iii), $\mathcal{S}'$ is a proper subset of $\mathcal{S}$.  Consider the top-most point $p\in\mathcal{S}'$ (which we can do since $\mathcal{S}'$ is a closed subset of a geodesic segment, identified, say, with the $z$-axis).  By iii), it follows that $p$ is not contained on the boundary of $B$.

We claim that any tangent cone $T_p\mathcal{V}$ to $\mathcal{V}$ at $p$ is the plane orthogonal to $\mathcal{S}$ at $p$, potentially with multiplicity.  To see this, consider any sequence of dilations $\mathcal{V}_j := \lambda_i\exp^{-1}_p(\mathcal{V}\cap B_1(p))\subset\mathbb{R}^3$ of the varifold $\mathcal{V}$ where $\lambda_i\rightarrow\infty$.  After a rotation, we may as well assume that the set $\mathcal{S}$ maps to the $z$-axis in $\mathbb{R}^3$. The surfaces $\mathcal{V}_j$ converge in $\mathbb{R}^3$ to a stationary cone $C$. Moreover, since $p$ was the top-most point of intersection of $\text{supp}(\mathcal{V})$ with the axis $\mathcal{S}$, it follows that the sequence $\mathcal{V}_j$ is disjoint from the positive $z$-axis $A^+$.  Let $A^-$ denote the negative $z$-axis.   Since the surfaces $\mathcal{V}_j$ have bounded genus and areas, classical results (\cite{CS2}) imply that $\mathcal{V}_j$ converge to $C$ smoothly away from finitely many points of $\mathbb{R}^3\setminus A^-$.  Thus it follows that $C$ cannot have any singular set away from the $z$ axis.  By  Lemma \ref{basic}, $C$ is a union of $k$ half-planes intersecting along the $z$-axis (potentially with different multiplicities).  If $k>2$, considering the convergence in a ball $R$ centered in $A^+$ but contained in the halfspace $\{z\geq 0\}$, one would have $\mathcal{V}_j\cap R$ a sequence of smooth minimal surfaces converging with bounded area and genus to a non-smooth surface, which is a contradiction again to \cite{CS2}.  It follows that $k=2$ and $C$ is a plane (potentially with multiplicity), which by Lemma \ref{actions} (as $n\neq 2$) must be orthogonal to $\mathcal{S}$.   By the Constancy Theorem, $C$ is an integer multiple of this plane. Since $\mathcal{V}$ is smooth, $G$-stable and thus stable by Proposition \ref{gstability} away from any cylinder about the singular set $\mathcal{S}$, it follows using the fact that the tangent cone at $x$ is an integer multiplicity plane orthogonal to $\mathcal{S}$ (see Step 4 in the proof of Proposition 6.3 in \cite{cd}) that the singularity at $p$ is removable.  This implies that there is an interval $I=[p,p']\subset\mathcal{S}$ so that $\mathcal{V}$ only intersects $I$ at $p$.  Thus we can iterate the above argument to obtain that $\mathcal{V}$ is smooth and intersects $\mathcal{S}$ orthogonally in a discrete set.  Note that when we iterate the argument, the points of intersection of $\mathcal{V}$ with $\mathcal{S}$ cannot accumulate at a limit point $p_\infty$.  Otherwise, at such a point $p_\infty$ the same analysis as above forces the tangent cone at $p_\infty$ to be a plane orthogonal to  $\mathcal{S}$.

It remains to prove regularity of $\mathcal{V}$ at the singular axis in the case $n=2$.  Because a tangent cone can contain the singular axis, the above argument is not sufficient.  In principle, if the minimizing sequence becomes tangent to the singular set, one can see by the $\mathbb{Z}_2$-equivarance that this should violate the maximum principle.  

We will prove that $\mathcal{V}$ has a smooth replacement in small equivariant balls centered about the singular axis.  By the replacement theory of \cite{pitts}, this implies that $\mathcal{V}$ is regular.  Note that Steps 1) and 2) in the following are very similar to the arguments of Almgren-Simon \cite{almgrensimon}, but we give full details so that one can see that the restriction to $\mathbb{Z}_2$ isotopies poses no problem.  In the following, let $\tau$ denote the generator of $\mathbb{Z}_2$.
\\ 
\\
\emph{Step 1: Reduction via neck-pinches}  

By Section 3 of Meeks-Simon-Yau \cite{msy}, we may perform finitely many $\gamma$-reductions (``neck-pinches") on the minimizing sequence $\Sigma_i$ (which we do not relabel) so that $\Sigma_i$ converges still to $\mathcal{V}$.  The reduced $\Sigma_i$ has the following key property: \emph{there exists $\epsilon>0$ so that for $i$ large enough, any closed Jordan curve $\alpha$ on $\Sigma_i$ of diameter at most $\epsilon$ bounds a disk in $\Sigma_i$}.

Note that in each $\gamma$-reduction in \cite{msy}, one is adding in two disks $D_1$, $D_2$ and removing a cylinder $C$.  The only additional consideration in the equivariant setting is the following: if $C$ is centered about $\mathcal{S}$ then only one neck-pinch is necessary.  If $C\cup D_1\cup D_2$ bounds a ball disjoint from $\mathcal{S}$, then one performs the neck-pinch twice (on $C\cup D_1\cup D_2$ and $\tau(C\cup D_1\cup D_2)$) in order to preserve $\mathbb{Z}_2$-equivariance.  This corresponds to the two types of permissible neck-pinches (ordinary and $\mathbb{Z}_k$, see Remark \ref{typeofneck}).
\\
\\
\noindent
\emph{Step 2: Reduction to disks}

Let $N$ be a uniformly convex $\mathbb{Z}_2$-ball centered around $x$ with radius at most $\epsilon$ from Step 1).  We also assume that $\Sigma_j$ intersect $\partial N$ transversally and that
\begin{equation}\label{ortho}
\mathcal{H}^2(\mathcal{V}\cap\partial N) = 0.
\end{equation}
 We claim that for $k$ large, we can perform an equivariant isotopy $\phi_t$ on $\Sigma_k$ so that $\phi_1(\Sigma_k)\cap N$ consists of a union of disks $D^k_1....D^k_j$ and 
\begin{equation}
\mathcal{H}^2(\phi_1(\Sigma_k))< \mathcal{H}^2(\Sigma_k).
\end{equation}

To produce the isotopy, first observe that because $N$ is uniformly convex, the following property holds:  if $F$ is a surface with boundary in $B\setminus N$ with boundary $\partial F$ contained in $\partial N$, and if $E$ is the smallest area set in $\partial N$ bounded by the curves $\partial F$, then 
\begin{equation}\label{uniformconvex}
\mathcal{H}^2(E) <  \mathcal{H}^2(F).
\end{equation}

To construct the purported isotopy, consider a component $F$ of the surface $\Sigma_k\cap (B\setminus N)$ that has boundary $\partial F$ entirely contained in $\partial N$.  Choose the component $\partial F^*$ of $\partial F$ so that the disk $D$ in $\Sigma_k$ bounded by $\partial F^*$ contains the other components of $\partial F$.   Note that by equivariance, either $F$ is interchanged with another component by the $\mathbb{Z}_2$ action, or else $\partial F^*$ is a circle centered around $\mathcal{S}$ (otherwise, it is easy to see that the connected component of $\Sigma_k$ containing $F$ is diffeomorphic to a two-sphere).  In the case that $F$ is interchanged, we repeat the procedure of the following paragraph for each of the two interchanged components, otherwise only once.

Replace the disk $D$ in $\Sigma_k$ by $E^*$ the disk in $\partial N$ that is bounded by $\partial F^*$, which decreases area by \eqref{uniformconvex}.  We can then press the disk $E^*$ inside $N$ changing the areas by an arbitrarily small amount.  In this way we can (equivariantly) reduce the number of components of $\Sigma_k\cap (B\setminus N)$ that have boundary $\partial F$ entirely contained in $\partial N$.  We can iterate this procedure until the only components of $\Sigma_k\cap (B\setminus N)$ have part of their boundaries among the boundary curves of $\Sigma$, $\{\gamma_i\}_{i=1}^k$. But such a component $C$ intersects $\partial N$ in circles, each of which bounds a disk in $N$ by Step 1).   Such a disk must be entirely contained in $N$ (otherwise, it would have been pushed into $N$ in the first stage of this process).
\\
\\
\emph{Step 3: Completion of argument}
\\
\indent
For each $k$ large, replace the disks $D^k_1,...D^k_j$ contained in $N$ by the area minimizers for their boundary curves in $\partial N$, to obtain new disks $\tilde{D}^k_1,...\tilde{D}^k_j$.  By Meeks-Simon-Yau \cite{msy}, such minimizers are smooth disks.  By Meeks-Yau (Proposition \ref{basic}.(2)), such disks are themselves equivariant.   It also follows from Meeks-Yau (Proposition \ref{basic}.(1)) that the minimizers are pairwise disjoint since their boundary curves are.  Note that each of the disks $\tilde{D}^k_j$ intersects $\mathcal{S}$ transversally, if it all, since otherwise, such a disk would have to contain the singular set $\mathcal{S}$ by Lemma \ref{intersect} in violation of the boundary regularity proved in \cite{almgrensimon} and \cite{dep}.  

By replacing the disks $D^k_1,...D^k_j$ in $\Sigma_k$ with $\tilde{D}^k_1,...\tilde{D}^k_j$ we obtain a new sequence $\tilde{\Sigma}_k$.  Note that since $\mathcal{H}^2(\tilde{\Sigma}_k) \leq \mathcal{H}^2(\Sigma_k)$, the sequence $\tilde{\Sigma}_k$ is still a minimizing sequence.   We claim that the limit $\mathcal{V}'$ of $\tilde{\Sigma}_k$ is a replacement for $\mathcal{V}$ in $N$.

We first show that $\mathcal{V}'$ coincides with $\mathcal{V}$ in $B\setminus N$ and thus $\mathcal{V}'$ satisfies (2) in the definition of replacement. If this fails, consider the components $C^k_i$ of $\Sigma_k\cap (B\setminus N)$ that were pushed into $N$ via Step 2.  Suppose the collection of $C^k_i$ contribute to the limit $\mathcal{V}$ in $B\setminus N$ for some subsequence of $k$'s (which we pass to without relabeling).  In this case, it follows that 
\begin{equation}\label{is}
\lim_{k \rightarrow\infty}\sum_i\mathcal{H}^2(C^k_i)\rightarrow \alpha>0. 
\end{equation}
\noindent
Thus (passing to a subsequence of $k$'s again) the area of the disks $E^k_i$ in $\partial N$ bounded by $C^k_i$ also satisfy 
\begin{equation}\label{mustbe}
\lim_{k\rightarrow\infty}\sum_i\mathcal{H}^2(E^k_i)\rightarrow \alpha.
\end{equation}
 Indeed, suppose \eqref{mustbe} failed.  Then since $\mathcal{H}^2(E^k_i)<\mathcal{H}^2(C^k_i)$ by \eqref{uniformconvex}, we obtain that
\begin{equation}\label{help}
\lim_{k\rightarrow\infty} \sum_i\mathcal{H}^2(E^k_i)=\beta <\alpha.
\end{equation}

By construction we have
\begin{equation}
\mathcal{H}^2(\tilde{\Sigma}_k)\leq \mathcal{H}^2(\Sigma_k)+\sum_i\mathcal{H}^2(E^k_i)- \sum_i\mathcal{H}^2(C^k_i),
\end{equation}
which together with \eqref{help} and \eqref{is} imply that
\begin{equation}
\lim\mathcal{H}^2(\tilde{\Sigma}_k)\leq \lim\mathcal{H}^2(\Sigma_k) +(\beta - \alpha).
\end{equation}

This contradicts the assumption that $\Sigma_k$ is a minimizing sequence.  Thus \eqref{mustbe} holds.  Since $\mathcal{H}^2(E^k_i)<\mathcal{H}^2(C^k_i)$ holds by construction, it follows from \eqref{mustbe} that for each $i$, $\mathcal{H}^2(C^k_i)\rightarrow\mathcal{H}^2(E^k_i)$ and also $C^k_i\rightarrow E^k_i$ in the flat topology (by the divergence theorem applied to a unit radial vector field defined outside of $N$.)
But then the varifold limit in $k$ of the surface $\cup_i C^k_i$ is supported on $\partial N$, contradicting \eqref{ortho}.  Thus we have shown that the limit of $\tilde{\Sigma}_k$ coincides with $\mathcal{V}$ in $B\setminus N$.  Since $\tilde{\Sigma}_k$ is still a minimizing sequence, its limit is moreover stationary. 

By Schoen's curvature estimate \cite{schoen}, the family of stable minimal disks $\tilde{D}^k_1,...\tilde{D}^k_j$ converges to a smooth and stable minimal surface (potentially with multiplicity).  Thus we have produced a smooth stable replacement for $\mathcal{V}$ in $N$. The replacement theory of \cite{pitts} (Proposition 6.3 in \cite{cd}) then implies that $\mathcal{V}$ is regular.

\end{proof}
This completes the proof of Theorem \ref{main}a,b. 
\indent
\section{Completion of Proof of Theorem \ref{main}}
Now we can complete the proof of e, f, and g in Theorem \ref{main}.   For e, it follows from Lemmas \ref{intersect} and \ref{intersect2} that the min-max limit (now proved smooth) must be orthogonal to part of the singular set with isotropy $\mathbb{Z}_n$ when $n\neq 2$.

For f, by Lemma \ref{intersect}, when $\Gamma$ is tangent to any part of the singular set with isotropy $\mathbb{Z}_2$ at a point $p$, it follows from Lemma \ref{intersect} that $\Gamma$ contains the singular set in a fixed neighborhood $N$ of $p$, which by iteration means $\Gamma$ contains the entire arc of constant isotropy containing $p$.  In this case, we claim the multiplicity of the component of $\Gamma$ containing $p$ is even.  By the replacement theory, $\Gamma$ is the limit in the interior of $N$ of $k$ stable graphs $f_1$, ... $f_k$ approaching $k\Gamma$ in $N$.  Note that by construction, none of these stable graphs contains the singular set $\mathcal{S}\cap N$.  Consider via the exponential map in $M$ that the functions $f_i$ are defined on $T_p\Gamma$, so that $\mathcal{S}\cap N$ is the $y$ axis in $T_p\Gamma$ and $\Gamma$ is a graph $G$ over the $xy$ plane containing $\mathcal{S}\cap N$ and satisfying $G(x,y) = -G(-x,y)$.  By equivariance, if $f_i$ is preserved by $\mathbb{Z}_2$, then $f_i(x,y)=-f_i(-x,y)$ implying $f_i(0,y)=0$, i.e., that $f_i$ contains the singular set $\mathcal{S}$ in $N$, which cannot occur.  Thus the action of $\mathbb{Z}_2$ interchanges $f_i$ with some other graph $f_{i'}$.  Since all graphs are interchanged, and since $\mathbb{Z}_2$ acts via diffeomorphism, it follows that $k$ is even.  This proves Theorem \ref{main}f and the analagous analysis implies Theorem \ref{main}g.
\subsection{Genus bounds: Proof of Theorem \ref{main}c,d}
In this section we indicate the necessary straightforward changes to the arguments of \cite{ketover} to prove that the min-max limit is achieved after equivariant surgeries.

For $\epsilon>0$ small enough, consider the tubular neighborhood $T_\epsilon(\Gamma)$.  By Theorem \ref{main}a,e,f,g. we obtain that the singular set $\tilde{\mathcal{S}}$ restricted to $T_\epsilon(\Gamma)$ consists of several arcs intersecting $\Gamma$ orthogonally with isotropy $\mathbb{Z}_n$, and may have some segments lying in the support of $\Gamma$ of isotropy $\mathbb{Z}_2$ that either close up or join each other at points of isotropy $\mathbb{D}_n$.

 It follows by the varifold convergence of $\Sigma_j$ to $\Gamma$ that there exists an $\epsilon>0$, so that $\Sigma_j$ intersects $\partial (T_\epsilon\Gamma)$ transversely in a union of small circles (see Proposition 2.3 in \cite{dep}).  Note that since the intersection points of $\Sigma_j$ with the singular locus are finite in number the circles comprising $\Sigma_j\cap\partial (T_\epsilon(\Gamma))$ are themselves disjoint from the singular locus for suitable $\epsilon$.  We cut along these circles to arrive at a new sequence $\tilde{\Sigma}_j$ that is now contained in a tubular neighborhood about $\Gamma$.  Some of these circles may be centered around $\tilde{\mathcal{S}}_1$, in which case we perform the surgeries $\mathbb{Z}_n$-equivariantly as $\mathbb{Z}_n$-neckpinches.  Other circles not centered around $\tilde{\mathcal{S}}_1$ will have $|G|$ copies and we do these ordinary neck-pinch surgeries isometrically for each copy.
In this way one produces a surgered min-max sequence (which we do not relabel) supported in a tubular neighborhood of the limiting minimal surface $\Gamma$.  It therefore suffices in the following to assume that $\Gamma$ is connected.  Suppose it occurs with multiplicity $n$.  For simplicity let us assume $\Gamma$ is orientable, as the non-orientable case follows with trivial modifications. 

In order to apply the Improved Lifting Lemma (\cite{smith}, \cite{ketover}) in the equivariant setting, one needs a good set of curves to consider and for this it is most useful to consider the projection of the min-max limit in the quotient orbifold.  To that end, consider $\pi(\Gamma)$ in $M/G$.  The surface $\pi(\Gamma)$ has some genus $g$ and potentially some piecewise smooth boundary curves $\{\alpha_i\}_{i=1}^j$ which consist of arcs of the singular set $\mathcal{S}$ of isotropy $\mathbb{Z}_2$.  The smooth pieces of the curves $\{\alpha_i\}_{i=1}^j$ join together at the $\mathbb{D}_n$ points.  

To justify the statement that $\pi(\Gamma)$ acquires boundary in this way when it contains a curve of isotropy $\mathbb{Z}_2$, it is enough to consider the local picture.  Thus consider $\mathbb{Z}_2$ acting on $\mathbb{R}^3$ by $180^o$ rotation $\tau$ about the $x$-axis.  The surfaces $\{z=0\}$ in the quotient $\mathbb{R}^3/\{\tau, e\}$ have boundary consisting of the $x$ axis.

Now let $\{\gamma_i\}_{i=1}^k$ be a collection of closed curves on $\pi(\Gamma)$ intersecting in one point such that $\pi(\Gamma)\setminus\cup_{i=1}^k\gamma_i$ is a topological disk $D$ with $j$ subdisks removed, each bounded by one of the closed curves $\alpha_i$.  We can assume that the curves $\{\gamma_i\}_{i=1}^k$ are disjoint from $\mathcal{S}$ since $\Gamma\cap\mathcal{S}$ consists of the curves $\{\alpha_i\}_{i=1}^j$ together with finitely many points $\mathcal{P}$.  As in Section 5 of \cite{ketover}, by cutting along the curves $\gamma_i$, we can identify $T_\epsilon(\Gamma)$ (in $\pi(M)$) with $G\times [-\epsilon,\epsilon]$, where $G$ is a regular polygon with several disks removed (those bounded by the curves $\alpha_i$).  We can add in several curves $\{\beta_i\}_{i=1}^l$ to $\pi(\Gamma)$ so that $G\setminus\{\beta_i\}$ consists of several disks $D_1$, ..., $D_r$, and each disk contains at most one point of $\mathcal{P}$.

The Improved lifting lemma (Proposition 2.2 in \cite{ketover}) implies that one can further surger the min-max sequence  so that it consists of the expected number of graphs in a neighborhood in $M$ of $\pi^{-1}(\cup_{i=1}^k\gamma_i)$.  This decomposition descends in $M/G$ to $n$ graphs along each closed curve $\gamma_i$. We may then apply the Improved lifting lemma to the closed curves $\pi^{-1}(\cup_{i=1}^j\alpha_i)$ as well as the segments $\pi^{-1}(\cup_{i=1}^l\beta_i)$.

Since each of the disks $D_i$ comprising $G$ has boundary among the curves $\alpha_i$, $\beta_i$ and  $\gamma_i$, 
it follows that the min-max sequence restricted to $D_i\times [-\epsilon,\epsilon]$ has boundary consisting of $n$ closed parallel curves in $\partial D_i\times [-\epsilon,\epsilon]$.  If $D_i$ contains no point of $\mathcal{P}$, by Lemma C.1 in \cite{dep} we can further surger the min-max sequence in $D_i\times [-\epsilon,\epsilon]$ so that it consists of $n$ parallel disks and all of these surgeries performed are ordinary neckpinches as they are supported away from the singular set by construction.  If however $D_i$ contains a point of $\mathcal{P}$ with isotropy $\mathbb{Z}_k$, we can similarly perform finitely many neck-pinches though some of these may be $\mathbb{Z}_k$-neckpinches.  

Since after these surgeries the min-max sequence in each disk $D_i$ consists of $n$ parallel disks, and the union of the disks comprise $G$, it follows that the min-max sequence has the desired decomposition in $G\times [-\epsilon,\epsilon]$ as $n$ parallel disks with boudaries parallel closed curves in  $\partial G\times [-\epsilon,\epsilon]$.  This implies the statement on degeneration (Theorem \ref{main}.c) and thus the genus bound Theorem \ref{main}.d.

This completes the proof of Theorem \ref{main}.

\section{Minimal surfaces in $\mathbb{S}^3$}
In this section, we will show that at least in $\mathbb{S}^3$, minimal surfaces with the same genus and symmetry group as many of the known examples can be constructed from Theorem \ref{main}.  Presumably our surfaces coincide with these examples but we cannot verify this.  

To the author's knowledge, the only known embedded minimal surfaces in $\mathbb{S}^3$ aside from great spheres and the Clifford tori are the following:
\begin{enumerate}
\item Desingularization of multiple great spheres (Lawson \cite{lawson} 1970)
\item Nine examples associated with tessalations of $\mathbb{S}^3$ by Platonic solids (Karcher-Pinkall-Sterling \cite{kps} 1988)
\item Doubling of Clifford torus (Kapouleas-Yang \cite{kapouleas} along square lattice 2003; Wiygul \cite{wigul1} along rectangular lattice 2013)
\item Desingularizing multiple Clifford tori along a geodesic (Choe-Soret \cite{choesoret} 2013)
\item Doubling of equator (Kapouleas \cite{kapouleas2} 2015)
\item ``Stacking" of multiple Clifford tori (Wiygul \cite{wigul1} 2015)
\end{enumerate}

We will first give a min-max interpretation of (1) and (2).  The surfaces (4) are constructed in the next section and are best understood as part of a much larger family of new minimal surfaces that we consider there.  The doublings (3) were considered in \cite{KMN}.  The surfaces (5) and (6) seem beyond are methods at present.  To double the equator, for instance, one needs to include in the symmetry group reflections through a great sphere.  Thus $M/G$ would have boundary, which we have explicitly excluded (see Remark \ref{whynoref}).  The ``stacking" of multiple Clifford tori due to Wiygul also seems difficult to construct using a min-max argument since one must consider at least $2$ parameter families, and ruling out multiplicity in this setting is challenging, even with the catenoid estimate.
\subsection{Computing the genus of equivariant minimal surfaces}

We recall the following Riemann-Hurwitz formula for branched covers.  Suppose $G$ acts on a three-manifold $M$.  Consider the projection map $\pi:M\rightarrow M/G$.  Suppose $\Sigma$ is an embedded surface in $M/G$ intersecting $\pi(\mathcal{S})$ transversally and consider the lifted surface $\tilde{\Sigma}=\pi^{-1}(\Sigma)$.   Then $f|_{\tilde{\Sigma}}:\tilde{\Sigma}\rightarrow\Sigma$ is a branched covering map. There are finitely many branch points $B$ in $\tilde{\Sigma}$.  For such a point $x\in\tilde{\Sigma}$, the ramification index is precisely $|G_{\pi(x)}|$, i.e. the order of the local isotropy group.  With this notation, the Riemann-Hurwitz formula is as follows:
\begin{equation}\label{riemann}
\chi(\tilde{\Sigma})=|G|\chi(\Sigma)-\sum_{x\in B} (|G_{\pi(x)}|-1).
\end{equation}

\subsection{Lawson surfaces}
We first identify $\mathbb{S}^3$ with a subset of $\mathbb{C}^2\equiv\mathbb{R}^4$:
$$\mathbb{S}^3=\{(z,w)\in\mathbb{C}^2\;|\;|z|^2+|w|^2=1\}.$$
We consider the group $G=\mathbb{Z}_{n+1}\times\mathbb{Z}_{m+1}$ acting on $\mathbb{S}^3$ as follows.   For any $(k,l)\in \mathbb{Z}_{n+1}\times\mathbb{Z}_{m+1}$ and $(z,w)\in\mathbb{S}^3$, we define the action to be $$(k,l).(z,w)=(e^{2k\pi i/(n+1)}z,e^{2l\pi i/(m+1)}w)$$   The singular set of $G$ acting on $\mathbb{S}^3$ consists of two linked circles: $$\mathcal{S}=\{z=0\}\cup\{w=0\}.$$ The isotropy group along the circle $\{z=0\}$ is $\mathbb{Z}_{n+1}$ and the isotropy group along the circle $\{w=0\}$ is $\mathbb{Z}_{m+1}$.  

Let us demonstrate that $M/G$ is topologically a $3$-sphere and that $M/G$ admits a sweepout by two-spheres where each sweepout surface $\Sigma/G$ intersects each of the two singular circles twice.   We may consider stereographic projection of $\mathbb{S}^3$ onto $\mathbb{R}^3$ based at the point $(z,w)=(0,-1)$.  In these coordinates the circle $C_1:=\{z=0\}$ maps to the $z$-axis of $\mathbb{R}^3$ and $C_2:=\{w=0\}$ maps to the unit circle in the $xy$-plane.  The great spheres in $\mathbb{S}^3$ containing the origin in $\mathbb{R}^3$ correspond to planes through the origin.

Assume $n\geq 2$ and $m\geq 2$.  Pick $Q_1,...,Q_{n+1}$ evenly spaced points along $C_2$ and $P_1, ... P_{m+1}$ along $C_1$.  The fundamental domain $\mathcal{F}$ for $G=\mathbb{Z}_{n+1}\times\mathbb{Z}_{m+1}$ acting on $\mathbb{S}^3$ is the convex hull of $P_1$, $P_2$, $Q_1$ and $Q_2$ (see for instance Figure 1 on page 347 in \cite{lawson}).  Since $M/G$ is formed from $\mathcal{F}$ by identifying opposite faces of the wedge we see that $M/G$ is an orbifold that is homeomorphic to a sphere.  

Let us sweep out $\mathcal{F}$ by genus $0$ surfaces $\{\Sigma_t\}_{t=0}^1$ as follows.  The surface $\Sigma_0$ consists of the planar region delimited by $P_1,P_2, Q_1$ together with the planar region delimited by $P_1,P_2, Q_2$.   Then for $t>0$, if $\{\gamma(t)\}_{t=0}^1$ parameterizes with constant speed the part of the curve $C_2$ between $Q_1$ and $Q_2$, set $\Sigma_t$ to be the union of the planes delimited by $P_1, P_2, \gamma(t)$ and $P_1, P_2, \gamma(1-t)$ restricted to $\mathcal{F}$.  Now desingularize $\Sigma_t$ along $C_1$ (the z-axis) to add in half of a fundamental domain of Scherk's singly periodic surface.  Then one can ``open the hole" of $\Sigma_t$ for $t$ near $1$ and $t$ near $0$ so that as $t\rightarrow 0$, $\Sigma_t$ approach the geodesics $P_1Q_1$, $P_2Q_1$ and $P_2Q_2$ and $P_1Q_2$.  Similarly, as $t\rightarrow 1$, we can open the hole so that $\Sigma_t$ converges to the geodesic segment joining the midpoint of $P_1P_2$ to the midpoint of $Q_1Q_2$ (see Figure 1 in \cite{lawson} for an illustration of most of these points, though note that his fundamental domain is smaller than ours).  The amended surfaces $\Sigma_t$ now intersect each singular circle $C_1$ and $C_2$ twice in $\mathcal{F}$.  Loosely speaking, given $2k$ planes intersecting at the same angle desingularized with a Scherk surface, there are two ways to press adjacent sheets together and then ``open up the holes," and our sweepout $\Sigma_t$ captures both of these directions.

The sweepout $\Sigma_t$ in $\mathcal{F}$ lifts to a $G$-sweepout of $\mathbb{S}^3$.  It remains to compute the genus of a lifted sweepout surface $\Sigma$ in $\mathbb{S}^3$.  One can see it directly but one can also apply the Riemann-Hurwitz formula \eqref{riemann} as follows.
 In each fundamental domain of $M$ acted on by $G$, $\Sigma$ is a two-sphere and intersects both the singular $n+1$ curve and the singular $m+1$ curve twice.  These lift to $2(n+1)$ points with ramification index $m+1$ and $2(m+1)$ points with ramification index equal to $n+1$.  Thus one obtains from \eqref{riemann}:
\begin{equation}
\chi(\tilde{\Sigma})=2(n+1)(m+1)-2(n+1)m-2(m+1)n=-2nm+2.
\end{equation}

Unraveling this we obtain $g(\tilde{\Sigma})=nm$, as expected.  
\\ \indent We now apply Theorem \ref{main} to our $G$-sweep-out by surfaces of genus $g=nm$, with $M=\mathbb{S}^3$ and $G=\mathbb{Z}_{n+1}\times\mathbb{Z}_{m+1}$ to obtain a minimal surface $\Gamma_{m,n}$.  It remains to compute the genus of $\Gamma_{m,n}$.   We know that $\Gamma_{m,n}$ must be connected because of the positivity of the Ricci curvature (it may however have multiplicity).  We claim that $\pi(\Gamma_{m,n})$ intersects each singular circle in $M/G$ in precisely two points.  By the calculation using the Riemann-Hurwitz formula above, this implies that the genus of $\Gamma_{m,n}$ is also $mn$, i.e., no degeneration occurred and we have produced a minimal surface with the same genus as the Lawson examples.  

Denote the singular circles in $M/G$ by $C_1$ and $C_2$, and we have that $\pi(\Sigma_j)$ intersects each of the circles twice.  To classify possible compressions, for any closed sphere $S$ in $M/G$ associate the ordered pair of non-negative integers $(a,b)$ where $a$ is the cardinality of $S\cap C_1$ and $b$ the cardinality of $S\cap C_2$.
We call the ordered pair the \emph{intersection type} of $S$.  We claim that if one starts with a surface with intersection type in the list $\mathcal{L}$ consisting of $(0,2)$, $(0,2)$, $(2,2)$ or $(0,0)$, then any surgery on such a surface gives components both of which have intersection type in $\mathcal{L}$.   

To check this, there are three types of neck-pinches: those centered around $C_1$, those centered around $C_2$ and those centered away from the singular locus.  Consider a neck-pinch of the first type, around $C_1$, and suppose one begins with a surface of type $(2,2)$ in  $\mathcal{L}$.  A neck-pinch about $C_1$ gives two components $S_1$  and $S_2$ and observe that $S_1$ and $S_2$ intersect $C_1$ a total of $4$ times (as the neck-pinch about $C_1$ adds two points of intersection in total).  Also, since the neckpinch adds a point of intersection to each component $S_1$ and $S_2$, we have that $\text{Card}(S_1\cap C_1)$ and $\text{Card}(S_2\cap C_1)$ are both at least $1$.  The mod $2$ intersection number of each $S_i$ with $C_1$ is $0$.  Thus $\text{Card}(S_2\cap C_1)$ and $\text{Card}(S_2\cap C_1)$ in fact must both be $2$.   Considering the intersections with the second circle $C_2$, we have either $\text{Card}(S_1\cap C_2)=2$ and $\text{Card}(S_2\cap C_2)=0$ or else $\text{Card}(S_1\cap C_2)=0$ and $\text{Card}(S_2\cap C_2)=2$ (again since the mod $2$ intersection number of the $S_i$ with $C_2$ is $0$).  In other words, the list $\mathcal{L}$ is preserved if one starts with a surface of type $(2,2)$ and does a $C_1$ or (mutatis mutandis) a $C_2$ surgery.  The third type of surgery supported away from the singular locus on $(2,2)$, gives either two spheres of type $(2,0)$ or $(0,2)$ or else  $(2,2)$  and $(0,0)$, again preserving the list $\mathcal{L}$.  Analagously one can show that starting at $(0,2)$, surgeries give components with intersection type still on the list $\mathcal{L}$. 

Since after surgeries, by Theorem \ref{main}c the components must be parallel, it follows that after surgeries on $\pi(\Sigma_j)$ one has $k$ parallel copies in $M/G$ of a surface of type $(2,2)$, $(0,2)$, $(2,0)$ or $(0,0)$.  But it follows from the Riemann-Hurwitz formula \eqref{riemann} that the surfaces $(0,2)$, $(2,0)$ or $(0,0)$ all lift to two spheres in $\mathbb{S}^3$.  None of these are equivariant when $\min(m,n)\geq 2$.  Thus in fact one has $k$ parallel copies of a surface of type $(2,2)$.  By the genus bounds (Theorem \ref{main}d) it follows that $k=1$.   This gives the Lawson surface.  Note that if  $m=1$ there is a danger that one obtains a great sphere with multiplicity $2$.  Using the ``catenoid estimate," one can likely rule this out.  See \cite{KMN}.

\begin{remark}\normalfont  One might wonder how we can produce Lawson's surfaces via min-max when Lawson himself used a minimization procedure in each fundamental domain.   Lawson actually first works in a symmetry group four times as big as ours, solves Plateau's problem for a polygon there, then reflects across an arc to fill in what our fundamental domain.  So each patch for him is the reflection of a stable minimal surface and so it is plausible that these have index 1 (half of a catenoid is stable but upon reflection it acquires index 1). \end{remark}

\begin{remark}\normalfont
Because our surfaces have less symmetry than those of Lawson, it is not clear that $\Gamma_{m,n}$ converge to a union of $m$ great spheres as $n\rightarrow\infty$.  Pitts-Rubinstein conjecture (see Section 1, Remark 6 in \cite{pr}) that this is true.  An interesting related question asked by Kapouleas (Queston 4.3 in \cite{kapdoub}) is whether one can produce desingularizations of equators that are not all meeting at an equal angle. It is a long-standing conjecture of Yau that $\mathbb{S}^3$ contains at most finitely many embedded minimal surfaces of a given genus (up to congruence) which would imply that the Lawson examples should not be able to ``flap their wings."
\end{remark}

\subsection{Tessellations by Platonic Solids: the Karcher-Pinkall-Sterling examples}
The nine examples of Karcher-Pinkall-Sterling \cite{kps} come from first considering tessellations of $\mathbb{S}^3$ by the Platonic solids.  For instance consider the binary icosahedral group $I^*$ with $120$ elements acting on $\mathbb{S}^3$.  The quotient $\mathbb{S}^3/I^*$  is the Poincar\'e dodecahedral space.  Acting on this space is the $60$ element group $I$ of orientation-preserving symmetries of a dodecahedron.   We set $M=\mathbb{S}^3/I^*$ and $G=I$.  We can construct a $G$-equivariant sweepout $\Sigma_t$ of $M$ by surfaces of genus $6$.   Let us construct these sweepouts explicity.  Recall there are $12$ faces of the dodecahedron.  Denote by $F_i$ the center of the $i$th face and $A_i, B_i, C_i,D_i, E_i$ the vertices on the $i$th face.  Let $C$ denote the center of the dodecahedron.   Let $\Sigma_0$ to be the graph $\mathcal{G}_1$ comprised of the arcs: $\cup_i CF_i$.  Then $\Sigma_t$ for $t$ small is defined to be the boundary of a tubular neighborhood of radius $t$ about $\mathcal{G}_1$.  One can extend this sweepout $I$-equivariantly until $t=1$ when it approaches in the Hausdorff topology the graph $\mathcal{G}_2$ consisting of the union in $i$ of the segments $A_iB_i$, $B_iC_i$, $C_iD_i$, $D_i,E_i$, $E_iA_i$ (i.e. the boundaries of the twelve faces of the dodecahedron).  We can then apply Theorem \ref{main} to obtain a smooth embedded $I$-invariant connected minimal surface $\Sigma$ in $M=\mathbb{S}^3/I^*$.   

It remains to compute the genus of $\Sigma$.   To do this, it helps to compute the isotropy groups.   A fundamental domain $\mathcal{F}$ for the $I$ action is the tetrahedron formed from the points $A_1$, $B_1$, $C$ and $F_1$.  It is topologically a three-sphere and the sweepout surfaces are two-spheres.  There is $\mathbb{Z}_2$ isotropy on the line connecting $F_1$ to the midpoint of the line connecting $A_1$ and $B_1$.  There are $30$ such points.   There is  $\mathbb{Z}_2$ isotropy on $A_1F_1$ and $B_1F_1$.  There are also $30$ such points.   There is also  $\mathbb{Z}_2$ isotropy on the line connecting $C$ to the midpoint of $A_1B_1$.  There are $30$ such points.  There is also $\mathbb{Z}_2$ isotropy along the line connecting $A_i$ and $B_i$, which is disjoint from the $\Sigma_t$. Finally there is $\mathbb{Z}_3$ isotropy along the line $A_1C$ and $B_1C$ of which there are $20$ such points (the number of vertices in the dodecahedron).  Thus the orbifold defect in the Riemann-Hurwitz formula \eqref{riemann} is $30+30+30+20(2)=130$, which is coherent as we obtain from \eqref{riemann} that $\chi(\Sigma_t)=60(2-2(0))-130 = -10$, as expected.

The singular locus $\mathcal{S}$ in fact is a tetrahedron with the four points in $\mathcal{S}_0$ as vertices, and $6$ edges (of type $\mathbb{Z}_2$, $\mathbb{Z}_2$,$\mathbb{Z}_2$,$\mathbb{Z}_3$, and $\mathbb{Z}_5$) as enumerated above.  The four vertex points in $\mathcal{S}_0$ consist of:  $F_1$ of type $(2,2,5)$, $C$ of type $(2,3,5)$, $A_i$ of type $(2,2,3)$, and the midpoint of $A_iB_i$ of type $(2,2,2)$.   

From Theorem \ref{main}.c we obtain that $\Sigma$ is achieved after compressions.  But we claim that any equivariant compression brings the genus down to $0$, or else pinches off spheres.  Thus $\Sigma$ would have to be a sphere with some multiplicity.  But non-trivial quotients of $\mathbb{S}^3$ admit no minimal spheres by Frankel's theorem.  Thus this cannot occur and the genus of $\Sigma$ is $6$.

Since $\pi(\Sigma_t)$ is a sphere, it bounds a three-ball on both sides of $M/G$, which is topologically a three-sphere. One can see this directly by visualizing the identifications that occur on the tetrahedron comprising the fundamental domain $\mathcal{F}$ for the group action. To classify the possible compressions it suffices to classify compressions into both such balls, $B_1$ and $B_2$.  The singular set $\mathcal{S}$ in $B_1$ is a figure ``H" where $\mathcal{S}\cap\partial B_1$ consists of the four boundary points of the $H$ (i.e, the ``H" is inscribed in the ball $B_1$). Again, one can see this by reading off how a piece of the sweepout surfaces restricted to $\mathcal{F}$ intersects the tetrahedron comprising $\mathcal{S}$ described above.

Recall that equivariant compressions come in two types, \emph{ordinary} or \emph{$\mathbb{Z}_k$-neckpinches}, where the latter are performed by cutting away an annulus centered around the singular set with isotropy $\mathbb{Z}_k$ and the former performed away from the singular set.   Thus the only possible $\mathbb{Z}_k$-neckpinches into $B_1$ divide the middle bar of the $H$ in half, or else is along the ``feet" or ``arms" of the $H$ and thus splits off a sphere in $M/G$ that intersects one of the singular arcs (either the ``feet" or ``arms"  of the $H$), twice.  This lifts to a collection of spheres.  Similarly, ordinary neckpinches can only pinch off spheres.

 Thus in fact no degeneration can occur and the genus of $\Sigma$ in $\mathbb{S}^3/I^*$ is $6$.  The minimal surface $\Sigma$ lifts to a surface $\Sigma'$ of genus 601 in $\mathbb{S}^3$ as can be seen from the formula $\chi(\Sigma')=120\chi(\Sigma)$.  

The eight other examples of Karcher-Pinkall-Sterling \cite{kps} arise analagously.

\begin{remark}\normalfont
As observed in \cite{pr} the above procedure also produces a genus $6$ surface in the Weber-Seifert hyperbolic space  (which arises by identifying the opposite faces in a dodecahedron with a different twist angle).   Instead of using Frankel's theorem to rule out degeneration to spheres, one uses the fact that manifolds with $\text{sec}=-1$ admit no minimal spheres.
\end{remark}

\begin{remark}\normalfont
The same argument gives a new construction of Schwarz genus $3$ surface in the standard flat three-torus.  Here one uses $O_{24}$ equivariant sweepouts and the collapse to a two-sphere is impossible because such a sphere would lift to many disjoint closed embedded minimal two-spheres in $\mathbb{R}^3$, an impossibility.
\end{remark}
 
\begin{remark}\normalfont
A related argument also gives several new examples of free boundary minimal surfaces in the unit $3$-ball $B$, provided one generalizes the equivariant theory to that setting.  Recall that a free boundary minimal surface is one intersecting $\partial B$ orthogonally.  Not many examples are known: the flat disk, the critical catenoid, surfaces of genus $0$ resembling two disks connected by many half-necks at the boundary discovered by Fraser-Schoen \cite{FS} (and produced by gluing by \cite{pacard}), and related examples of genus $1$ by \cite{pacard}.  To construct new examples, consider $O_{24}$ acting on $B$.  There's an obvious sweepout of $B$ by surfaces of genus $0$ starting at the graph consisting of the union of the $x$ $y$ and $z$ axes in $B$ and ending at the tessellation of $\partial B$ by $6$ squares.  The surfaces in the sweepout have genus $0$ and $6$ ends.  Running an equivariant min-max procedure, if any degeneration occurs one obtains a union of free boundary minimal disks.  But no such disk is $O_{24}$-equivariant.  It follows that one obtains a free boundary minimal surface of genus $0$ and $6$ ends resembling a three dimensional ``cross."  One can also obtain in this way genus $0$ free boundary minimal surfaces associated to the the group $T_{12}$ with $4$ ends, and the group $I_{60}$ with $12$ ends.  One can think of these examples as the free boundary analog to the surfaces of Karcher-Pinkall-Sterling \cite{kps}.  On a related note, recently Li and Zhou \cite{lz} have obtained the existence of infinitely many free boundary minimal surfaces contained in convex domains, though one does not know their topology.
\end{remark}

\subsection{New Minimal Surfaces in $\mathbb{S}^3$}
We give a construction of eight of the families of the minimal surfaces proposed by Pitts-Rubinstein \cite{pr}. The following examples are best visualized in terms of the Hopf fibration. Consider $\mathbb{S}^3\subset\mathbb{C}^2$ 
\begin{equation}
\mathbb{S}^3=\{(z,w)\in\mathbb{C}^2\;|\; |z|^2+|w|^2=1\}.
\end{equation}
Recall the Hopf map $H:\bS^3\rightarrow\mathbb{C}\cup\{\infty\}$ given by
\begin{equation}
H(z,w)= z/w.
\end{equation}
Identifying $\mathbb{C}\cup\{\infty\}$ with the Riemann sphere $\mathbb{S}^2$ via stereographic projection, we can think of $H$ as a map from $\mathbb{S}^3$ to $\mathbb{S}^2$.

Let $\mathcal{T}$ be a geodesic net in $\bS^2$.  A geodesic net is a set of geodesic segments so that the tangent vectors at intersection points sum to $0$.   Such a net is a stationary $1$-varifold in $\bS^2$. The lifted $\bS^1$-invariant surface $H^{-1}(\mathcal{T})$ is then also a stationary $2$-varifold in $\bS^3$. Points in $\bS^2$ lift via the Hopf map to closed geodesics in $\bS^3$ and the equators of $\mathbb{S}^2$ lift to Clifford tori in $\bS^3$. In the following examples, we will insert genus (i.e. Karcher-Scherk towers) along the geodesics given by $H^{-1}(\mathcal{P})$ where $\mathcal{P}$ are the non-smooth points of the net $\mathcal{T}$.   In this way we can desingularize and double surfaces of the form $H^{-1}(\mathcal{T})$ for many different geodesic nets $\cT$.

\subsection{Doublings of stationary varifolds}\label{doublingstat}
Let $\mathcal{T}_1$ be the spherical dodecahedral tessellation (and geodesic net) of $\bS^2$ consisting of twelve curved pentagons suitably pieced together.  Consider the stationary varifold $H^{-1}(\mathcal{T}_1)$ in $\bS^3$.  Our surfaces will consist of ``doublings" of this stationary configuration.  Denote by $P_i$ the 12 pentagons in $\mathcal{T}_1$.  Each $P_i$ lifts to a piecewise smooth torus in $\bS^3$.  

Consider sweepouts $\Gamma_t$ of $\bS^2$ that are equivariant for the dodecahedral group of rotations in $SO(3)$.  Let $\Gamma_0$ consist of the centers of each pentagon.  As $t$ incrases, let $\Gamma_t$ consist of twelve equivariant shapes in each pentagon which at time $t=1$ is equal to the $1$-skeleton of the tessellation (with multiplicity $2$).  Now lift this sweepout to $\bS^3$ via the Hopf map to obtain a sweepout of $\bS^3$: $\Sigma_t=H^{-1}(\Gamma_t)$.  The surfaces $\Sigma_t$ consist of $12$ disconnected tori but we can join them via a Karcher-Scherk tower (see for instance \cite{weber}) along the $20$ geodesics that are the lifts of the vertices $\cT_1$ to $\bS^3$.  Then one can ``open up the necks" as in the Lawson examples to collapse to a union of geodesic segments.  It is clear from the construction that one can obtain such a sweepout by surfaces with areas exceeding $2\mathcal{H}^2(H^{-1}(\mathcal{T}_1))$ by an arbitrarily small amount. (In fact by the catenoid estimate \cite{KMN} one can obtain more precise estimates for the areas in the sweepout, but such estimates are not necessary.  The reason is that in doubling a smooth surface, one always has to rule out obtaining the surface with multiplicity $2$, while in doubling a stationary varifold, there is no danger that one obtains the stationary varifold with multiplicity $2$ by the regularity for equivariant min-max limits proved in this paper).

We want to impose an additional cyclic symmetry as we move along the Hopf fibers.   Recall from the classification of Hopf \cite{hopf} that $I^*\times\mathbb{Z}_{n}$ acts freely on $\mathbb{S}^3$ when $n$ has no $2$, $3$ or $5$ in its prime factorization, ie. $\text{gcd}(n,30)=1$.   Here the action of $I^*$ on $\mathbb{S}^3$ is on the left and it permutes the $120$ dodecahedra making up $\mathbb{S}^3$.  The $\mathbb{Z}_n$ acts on the right by moving along the Hopf fibers.   Now consider the action $I^*\times\mathbb{Z}_{3n}$ for some $n$ with $\text{gcd}(n,30)=1$.  Consider the element $e\in I_{60}\subset SO(3)$ of order $3$ that fixes a vertex $v_e\in\mathcal{T}_1$ and rotates about it.  Let $e^*$ be the lift of $e$ to $I^*$.  Since $e^*$ acts freely on $\mathbb{S}^3$ and fixes the fiber $H^{-1}(v_e)$, restricting to this fiber it must act as a translation by $1/3$ of the length of the fiber.   Thus for the action $I^*\times\mathbb{Z}_{3n}$ on $\mathbb{S}^3$ there is $\mathbb{Z}_3$ isotropy along the $20$ closed geodesics comprising the lifts of the vertices of the pentagonal tessellation $\mathcal{T}_1$.  Applying Theorem \ref{main} we obtain a connected minimal surface $\Gamma_n$. If any of the genus were to collapse when $n$ is large, arguing as in the previous two examples, the surfaces become a disconnected union of tori or spheres.  Thus one would obtain a torus or sphere (potentially with multiplicity) but neither the Clifford tori nor great spheres are symmetric with respect to the given group.  Thus no degeneration can occur for large $n$.

We claim $\Gamma_n\rightarrow 2H^{-1}(\mathcal{T}_1)$ as varifolds as $n\rightarrow\infty$.  To see this, observe that $\lim_{n\rightarrow\infty}\Gamma_n$ is invariant under translation along the Hopf fiber.  Thus it is supported on the lift  $H^{-1}(\mathcal{T})$ for some stationary varifold $\mathcal{T}$ supported on $\mathbb{S}^2$.  Since each $\Gamma_n$ contains intersection points with the lifts of $\mathcal{T}_1$, it follows from the monotonicity formula that $\mathcal{T}$ contains these $20$ vertices.  Since all genus is being concentrated and collapsing along the lifts of the twenty vertices, it follows that $\mathcal{T}$ has no other non-smooth points aside from the twenty vertices of the dodecahedron.  Thus $\mathcal{T}$ consists of these $20$ points together with the geodesic segments joining them.  In other words $\mathcal{T}=\mathcal{T}_1$.  It remains to show that $\Gamma_n$ converges to $H^{-1}(\mathcal{T}_1)$ with multiplicity $2$.  It follows from the equivariance that the multiplicity is an even integer $k$.  Since as observed above, $\Gamma_n$ has area at most exceeding $2\mathcal{H}^2(H^{-1}(\mathcal{T}_1))$ by an arbitrarily small amount, it follows that $k=2$.

The same procedure can be carried out for other tessellations of $\mathbb{S}^2$.  Namely, let $\mathcal{T}_2$ be the icosahedral tessellation, $\mathcal{T}_3$ the octahedral tessellation (consisting of three pairwise orthogonal closed geodesics), and $\mathcal{T}_4$ the tetrahedral tessellation.    These four examples comprise Examples 4, 5, 7, and 9 in Table 1 from \cite{pr}.  

\subsection{Desingularizations of stationary varifolds}
Consider the unique truncated cube where the added triangular faces are chosen to meet all of their neighbors at a point.  It consists of $8$ triangles and $6$ squares.  Such a solid gives rise to a geodesic net $\cT_5$ in $\bS^2$.  It has octohedral symmetry.  Now consider a sweepout of $\bS^2$ where $\Gamma_0$ consists of the center of each square.  As $t$ increases, let $\Gamma_t$ be (disjoint) concentric circles of increasing radius.  The radius can increase and then as $t\rightarrow 1/2$, let the circles collapse into the one-skeleton of $\cT_5$.  For $t>1/2$ push $\Gamma_t$ equivariantly into the {\it triangles} and then let the radius decrease so that by $\Gamma_1$ the surfaces collapse to consist of the center of each triangle.  This gives an equivariant sweepout of $\bS^2$.   Then the lifted surfaces $\Sigma_t:=H^{-1}(\Gamma_t)$ give rise to an equivariant sweepout of $\bS^3$ where one connects the various disconnected pieces via Scherk like towers along the lifts of the vertices in $\cT_5$.  As in the previous examples, to ensure symmetry about the lifts of the vertices of the tessellation, we enforce $O^*\times \mathbb{Z}_{2m}$ symmetry, where $m$ has no factors in its prime decomposition of $2$ and $3$.  Applying Theorem \ref{main} one obtains a minimal surface $\Gamma_m$.  As in the previous example no degeneration is possible when $m$ is large and one can see that $\Gamma_m\rightarrow H^{-1}(\mathcal{T}_5)$ in the sense of varifolds as $m\rightarrow\infty$.

One can perform the same operation with the truncated dodecahedron and truncated tetrahedron giving Example 3 and Example 8 in the table in \cite{pr}.

\subsection{Choe-Soret examples: Desingularization of multiple Clifford tori}
This example is much like the previous one, only here the configuration desingularized is a smooth surface rather than a singular stationary varifold.  Consider $\cT_6$ consisting of $k$ evenly spread out longitudes meeting in $2$ antipodal points.  
The sphere is cut into $2k$ sectors.  This configuration has dihedral symmery $\mathbb{D}_{2k}$.  Let $\Gamma_0$ be the centers of the even sectors.  As $t$ increases let $\Gamma_t$ be equivariant closed curves in the even sectors approaching $\cT_6$, until $t=1/2$, where $\Gamma_t$ coincides with $\cT_6$.  Then for $t>1/2$ push the curves into the \emph{odd} sectors and let the closed curves shrink to points at $\Gamma_1$ in the odd sectors.   Again by adding in tubes equivariantly and applying Theorem \ref{main} one produces minimal surfaces $\Gamma_n$ converging to $H^{-1}(\mathcal{T}_6)$ with multiplicity $1$.


\subsection{Further discussion}
Given a closed Riemannian $3$-manifold $M$ and $\Lambda >0$, consider the moduli space: $$\mathcal{M}_\Lambda(M)=\{\Sigma\subset M \;|\; \Sigma \text{ an embedded minimal surface with } \mathcal{H}^2(\Sigma)\leq \Lambda\}.$$  An interesting problem is to describe the boundary of this moduli space - in other words, limits of sequences of surfaces in $\mathcal{M}_\Lambda(M)$ when the genus is unbounded (if the genus is bounded such limits are understood by classical results).  For instance, the Lawson surfaces $\xi_{1g}\subset \mathcal{M}_{8\pi}(\bS^3)$ converge as  $g\rightarrow\infty$ to a union of two great spheres intersecting orthogonally.  A sequence of surfaces in $\mathcal{M}_\Lambda(M)$ (after passing to a subsequence) must converge to a stationary varifold, but not much more is known about what limits can arise.   

The new minimal surfaces in $\bS^3$ we have constructed give interesting types of behavior that can occur in the limit.  In previous desingularizations, the limiting varifold seems to be a union of embeddings.  The doublings of stationary varifolds (Example \ref{doublingstat}) show that one cannot even hope the limit will be a union of minimal immersions.  Already in $\mathbb{R}^3$ by blowing down multi-ended Karcher-Scherk surfaces, one can produce a sequence of smooth minimal surfaces converging to a stationary varifold that is not a union of smooth immersions, but Example \ref{doublingstat} may be the first known such example in a closed $3$-manifold.   

The only general result about the kinds of stationary varifolds that can arise is due to Brian White \cite{white}.  He proved that given a sequence $\Sigma_i\rightarrow V$ of minimal surfaces with bounded area, the associated mod 2 flat chains $[\Sigma_i]\rightarrow [V]$ also converge.  In the case of doublings above, the limit $V$ is a stationary varifold with multiplicity $2$ so that $[V]$ is $0$.  Indeed, the doubling surfaces converge in the flat topology to $0$ (they bound a region of arbitrarily small volume).  Thus while ``triple junctions" with multiplicity 1 cannot arise as a limit of orientable surfaces, they can arise with even multiplicity.

\bigskip
\end{document}